\documentclass[14pt,a4paper,reqno]{amsart}
\usepackage{vmargin,color}
\usepackage[latin1]{inputenc}
\usepackage{amsmath,amsfonts,amsthm,epsfig,graphicx}
\usepackage{pstricks,pst-plot,multido}
\usepackage{caption}
\usepackage{esint} 
\usepackage{pgfplots}
\usepackage{graphicx,tikz,subfigure,color,calc}
\usetikzlibrary{decorations.pathreplacing}

\pgfplotsset{compat=1.10}
\usepgfplotslibrary{fillbetween}
\usepgfplotslibrary{polar}
\usetikzlibrary{patterns}
\usepackage{relsize}

\usepackage{bigints}

\definecolor{darkgreen}{rgb}{0.0,0.5,0.0}
\definecolor{darkblue}{rgb}{0.0,0.0,0.3}
\definecolor{nicosred}{rgb}{0.65,0.1,0.1}
\definecolor{light-gray}{gray}{0.6}
\definecolor{really-light-gray}{gray}{0.8}

\usepackage{marginnote}

\def\R{\mathbb R}

\def\k{\kappa}

\newcommand{\diver}{\mathrm{div}} 
\newcommand{\mres}{\mathbin{\vrule height 1.6ex depth 0pt width  
0.13ex\vrule height 0.13ex depth 0pt width 1.3ex}}    

\def\big{\bigskip}

\newtheorem{theorem}{Theorem}[section]
\newtheorem{remark}[theorem]{Remark}

\newtheorem{proposition}[theorem]{Proposition}
\newtheorem{lemma}[theorem]{Lemma}

\newtheorem{example}[theorem]{Example}

\numberwithin{equation}{section}
\numberwithin{figure}{section}

\begin{document}

\title{Perimeter inequality under circular and Steiner symmetrisation: geometric characterisation of extremals}

\author[M. Perugini]
 {Matteo Perugini}
 \address[Matteo Perugini]{Universit\'a degli Studi di Milano, Dipartimento di Matematica, Milano}
 \email{matteo.perugini@unimi.it}

\maketitle

\begin{abstract}
{\rm We study the perimeter inequality under circular symmetrisation, and we provide a full geometric characterisation of equality cases. A careful inspection of the proof shows that a similar characterisation holds true also for the perimeter inequality under Steiner symmetrisation. Our result is based on a new short proof of the perimeter inequality under symmetrisation.}
\end{abstract}

\section{Introduction}
In this paper we give a geometric characterisation of the extremals of circular and Steiner perimeter inequalities. 
\subsection{Overview}
Symmetrisation procedures have been proved to be important tools in mathematical analysis, and indeed they have been widely used to deduce geometric properties of minimizers of  variational problems, and of solutions to PDEs. For instance, Steiner symmetrisation is a fundamental instrument in the proof by Ennio De Giorgi of the isoperimetric inequality (see \cite{DeGiorgi58ISOP,DeGiorgiSelected}, and \cite[Chapter 14]{maggiBOOK}), while Schwarz symmetrisation was used to prove the classic Faber-Krahn inequality (see \cite[Chapter II.8]{kawohl_book_85}).

Despite these techniques have been used for many decades, the detailed study of the equality cases for the perimeter inequalities under symmetrisation, is a relatively recent topic of investigation. One of the first results in this direction is due to De Giorgi: in his proof of the isoperimetric inequality he showed that if a set satisfies equality in Steiner's inequality, then it must be convex along the direction in which one performs the symmetrisation. After that, the problem of characterizing the equality cases for Steiner's inequality was resumed, and intensively investigated by Chleb\'ik, Cianchi and Fusco. In their seminal work \cite{ChlebikCianchiFuscoAnnals05}, the authors gave necessary conditions that a set must satisfy in order to be an extremal \cite[Theorem 1.1]{ChlebikCianchiFuscoAnnals05}, and provided sufficient conditions under which \emph{rigidity} of equality cases holds true. Here, by \emph{rigidity} we mean the case where the only sets achieving equality are those that are already symmetric (w.r.t. the symmetrisation procedure under consideration). The results obtained in \cite{ChlebikCianchiFuscoAnnals05} were successfully extended to the Steiner symmetrisation in any codimension in \cite{barchiesicagnettifusco} (thus including also the Schwarz symmetrisation), but still no full characterisation of the cases of equality was proved. 

Finally in \cite{CagnettiColomboDePhilippisMaggiSteiner}, Cagnetti, Colombo, De Philippis and Maggi gave a full analytic characterisation of equality cases for the Steiner's inequality in terms of the properties of the barycenter function (see \cite[Theorem 1.9]{CagnettiColomboDePhilippisMaggiSteiner}). Thanks to new tools introduced by the same authors in \cite{ccdpmGAUSS}, they were able to further push the study of the \emph{rigidity}, obtaining new important results. Still in the framework of Steiner symmetrisation, inspired by \cite{CagnettiColomboDePhilippisMaggiSteiner}  and employing some general notions of convex analysis, the author was able to extend the analytic characterisation of equality cases to  the anisotropic setting (see \cite[Theorem 1.8]{Perugini}). 

Despite a full characterisation of equality cases was successfully achieved for the Steiner's inequality, for other types of symmetrisation procedures such result is still missing. In particular, in the aforementioned work presented in \cite{ccdpmGAUSS}, the authors were able to fully characterize the \emph{rigidity} of equality cases for the Gaussian perimeter inequality under Ehrhard's symmetrization, but they only showed useful necessary conditions (not sufficient) for equality cases (see \cite[Theorem A]{ccdpmGAUSS}). Lastly, a similar situation to the one just described for the Gaussian perimeter was obtained but in the setting of the perimeter inequality under spherical symmetrisation. Indeed in \cite{CagnettiPeruginiStoger} the author together with Cagnetti and St\"oger were also able to provide the full characterisation of the \emph{rigidity} problem, but regarding the characterisation of extremals nothing more than a result that can be considered as the spherical counterpart of \cite[Theorem 1.1]{ChlebikCianchiFuscoAnnals05} was achieved (see \cite[Theorem 1.1]{CagnettiPeruginiStoger}).

The analytic characterisations given in \cite[Theorem 1.9]{CagnettiColomboDePhilippisMaggiSteiner} and \cite[Theorem 1.8]{Perugini} have proven to be quite helpful in the study of \emph{rigidity}. However, they can be quite difficult to use in specific situations, since they are expressed in terms of fine properties of the barycenter function of the one dimensional slices of the sets.

In this paper we present a geometric characterisation of extremals for the perimeter inequality under both circular and Steiner symmetrisation. Such characterisation is written in terms of geometric properties of the (measure-theoretic) inner unit normal $\nu^E$ to the set $E$ to which the symmetrisation is applied (see Theorem \ref{Eq. cases per. ineq.}). In the Steiner setting, these properties appear easier to check than the analytic conditions given in \cite[Theorem 1.9]{CagnettiColomboDePhilippisMaggiSteiner}.
In the framework of circular symmetrisation, to the best of our knowledge this is the first characterisation result for the extremals of the perimeter inequality.

We will provide a detailed proof of our result for circular symmetrisation, and we will then show how this can be adapted to the Steiner setting. Inspired by \cite[Section 4.1.5]{GMSbook1}, we introduce a measure associated to the distribution function of the set under consideration (see \eqref{ def of sigma}). This allows us to give a short and direct proof of the perimeter inequality and, in turn, to describe the extremals.

As far as we know, the circular symmetrisation for sets, and its application to rearrangements of functions, was firstly introduced by P\'olya in \cite{polya50} (see also \cite[A.7--A8]{polyaszego51}, and \cite[Chapter II.9]{kawohl_book_85}). Let us now precisely introduce the circular symmetrisation for sets (see also Figure \ref{fig: sym circ}).
\subsection{Circular symmetrisation for sets}
Let us start presenting some of the notation we will use in this paper. Let $k \in \mathbb{N}$, with $k \geq 2$.
We will decompose $\mathbb{R}^k$ as $\mathbb{R}^2 \times \mathbb{R}^{k-2}$, 
and we will write $(x, z) \in \mathbb{R}^k$, with $x \in \mathbb{R}^2$ and $z \in \mathbb{R}^{k-2}$. We will denote by $| \cdot |$ the Euclidean norm of 
$\mathbb{R}$, $\mathbb{R}^2$, $\mathbb{R}^{k-2}$, $\mathbb{R}^k$, or the total variation of a Radon measure, depending on the context. For $d\in \mathbb{N}$, with $1\leq d\leq k$ we denote by $\mathcal{H}^d$, and $\mathcal{L}^d$ the $d$-dimensional Hausdorff and Lebesgue measure in $\mathbb{R}^k$, respectively. We set $\mathbb{R}^2_0 := \mathbb{R}^2 \setminus \{ (0, 0) \}$, $\mathbb{S}^{1} = \{ x \in \mathbb{R}^{2}_0 : |x| = 1 \}$, and $\mathbb{S}^{k-1} = \{ (x,z) \in \mathbb{R}^{k} : |(x,z)| = 1 \}$. Moreover, given $r>0$ we write $\partial B(r)=\{x\in \mathbb{R}^2:\, |x|=r  \}$ to denote the boundary of the $2$-dimensional ball centered at the origin with radius $r$. Lastly, for every $x\in\mathbb{R}^2_0$ we set $\hat{x}=x/|x|$.
\medskip

We are now going to define the circular symmetral of a Borel set in $\mathbb{R}^k$ with respect to the half-hyperplane 
\( \{ (x_1, x_2, z_1, \ldots, z_{k-2}) \in \mathbb{R}^{k} : x_1 > 0, x_2 = 0 \}  = (0, \infty) \times \{ 0 \} \times 
\mathbb{R}^{k-2} \). For every Borel set $E \subset \mathbb{R}^{k}$ we define
\begin{align}\label{def: circular slice}
E_{(r, z)} := \{ x \in \mathbb{R}^2_0 : |x| = r  \text{ and } (x,z) \in E \} \subset \partial B(r)
\qquad \text{ for every } (r, z) \in (0, \infty) \times \mathbb{R}^{k-2}.
\end{align}
Note that, by definition, we have 
\[
0\leq \mathcal{H}^1 (E_{(r, z)}) \leq 2 \pi r, \quad \text{ for every } (r, z) \in (0, \infty) \times \mathbb{R}^{k-2}.
\]
Let now $\mu : (0, \infty) \times \mathbb{R}^{k-2} \to [0, \infty)$ be a Lebesgue measurable function satisfying 
\begin{equation} \label{compatibility for mu}
0 \leq \mu (r, z) \leq 2 \pi r, \quad \text{ for $\mathcal{L}^{k-1}$-a.e. } (r, z) \in (0, \infty) \times \mathbb{R}^{k-2}.
\end{equation}
We will say that $E$ is \textit{$\mu$-distributed} if 
\[
\mu (r, z) = \mathcal{H}^1 (E_{(r, z)}), \quad \text{ for $\mathcal{L}^{k-1}$-a.e. } (r, z) 
\in (0, \infty) \times \mathbb{R}^{k-2}.
\]
Given a Lebesgue measurable function $\mu : (0, \infty) \times \mathbb{R}^{k-2} \to [0, \infty)$ satisfying \eqref{compatibility for mu}, 
we define the set $F_{\mu}\subset \mathbb{R}^k$ as 
\begin{align}\label{def: F_mu}
F_{\mu} := \left\{ (x, z) \in \mathbb{R}^2_0 \times \mathbb{R}^{k-2} \, :  
 2 |x| \arccos(\hat{x}\cdot e_1) < \mu ( | x | , z) \right\},
\end{align}
where $e_1\in \mathbb{R}^2$ is defined as $e_1=(1,0)$.
\begin{figure}[!htb]
\centering
\def\svgwidth{15cm}
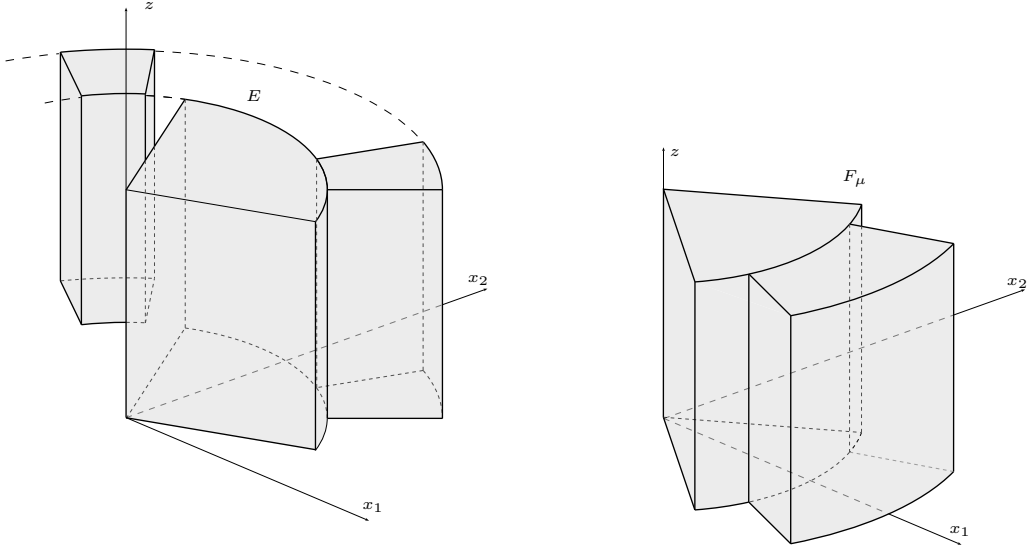
\caption{A pictorial representation in $\mathbb{R}^3$ of a $\mu$-distributed set $E$ and of its circular symmetral $F_\mu$.}
\label{fig: sym circ}
\end{figure}

\begin{remark} \label{rem used for monotonicity}
Note that by definition of $F_{\mu}$, we have 
\[
(x, z) \in F_{\mu} \, \Longrightarrow \,
(w, z) \in F_{\mu} 
\quad \forall \, \text{$w$ with $|w| = |x|$ and $\arccos(\hat{w}\cdot e_1)
\leq \arccos (\hat{x}\cdot e_1)$.}
\]
\end{remark}
\noindent
If $E \subset \mathbb{R}^{k}$ is a $\mu$-distributed Borel set, 
we say that $F_\mu$ is the \textit{circular symmetral} of $E$
with respect to the half-hyperplane 
\( \{ (x_1, x_2, z_1, \ldots, z_{k-2}) \in \mathbb{R}^{k} : x_1 > 0, x_2 = 0 \}\).

\medskip
\noindent
There is a particular bond between circular symmetrisation and Steiner symmetrisation. Firstly, both symmetrisation techniques act by slicing sets with lines of dimension 1, and second, as observed by P\'olya and Szeg\"o themselves, the limit as $c\to -\infty$ of the circular symmetrisation of a set $E$ w.r.t. the half-hyperplane 
\( \{ (x_1, x_2, z_1, \ldots, z_{k-2}) \in \mathbb{R}^{k} : x_1 > c, x_2 = 0 \}\)  ``tends'' to the Steiner symmetrisation of $E$ w.r.t. the full hyperplane \( \{ (x_1, x_2, z_1, \ldots, z_{k-2}) \in \mathbb{R}^{k} : x_2 = 0 \} \).

\subsection{Main results}
Let us now present the main results of this  work. The first result consists in a precise description of the geometric properties of the symmetral set $F_\mu$ defined in \eqref{def: F_mu}. In the following, given any set $E\subset \mathbb{R}^k$ of locally finite perimeter we denote with $\partial^*E$ and with $\nu^E(x,z)$ the reduced boundary of $E$, and the (measure-theoretic) inner unit normal to $\partial^*E$ at $(x,z)$, respectively (see Section \ref{preliminaries} for more details). Given any vector $\nu\in \mathbb{R}^k$ we set   
\begin{align*}
\nu_{  x } = (\nu_1, \nu_2),\qquad
\nu_z = (\nu_3, \nu_4, \ldots, \nu_k).
\end{align*}
In particular,  given any vector field $\nu: \mathbb{R}^k\to \mathbb{R}^k$ we use the following notation:    
\begin{align}\label{eq: circular decomposition of vectors}
\nu_{  x {\scriptscriptstyle\perp}}(x,z) = (\hat{x} \cdot \nu_x(x,z))\hat{x}, \qquad
\nu_{  x {\scriptscriptstyle\parallel}}(x,z) = \nu_x(x,z)-\nu_{  x {\scriptscriptstyle\perp}}(x,z)\quad \forall\,(x,z)\in \mathbb{R}^2_0\times\mathbb{R}^{k-2}.
\end{align}
Given $E\subset \R^k$ set of locally finite perimeter, we set
\begin{align}\label{eq: the circular version of nu^E}
\nu^E_{\mathsf{c}}(x,z):=(\hat{x}\cdot\nu^{E}_x (x,z),|\nu^{E}_{\!  x {\scriptscriptstyle\parallel}}(x,z)|,\nu^{E}_z(x,z)),\quad
\textnormal{for $\mathcal{H}^k$-a.e. $(x,z)\in\partial^*E \cap (\mathbb{R}^2_0\times\mathbb{R}^{k-2})$}.
\end{align}
\noindent
First of all, we show some useful symmetry properties of the (measure-theoretic) inner unit normal $\nu^{F_\mu}$ of $F_\mu$.
\begin{proposition} \label{real magic property}
Let $\mu  : (0, \infty) \times \mathbb{R}^{k-2}\to [0, \infty)$ be a Lebesgue measurable function
satisfying \eqref{compatibility for mu} such that $F_\mu$ is a set of locally finite perimeter. Then, \textbf{for every} $(r, z) \in (0, \infty) \times \R^{k-2} $
such that $( \partial^* F_{\mu} )_{(r,z)} \neq \emptyset$, the functions  
\begin{align}\label{eq: the components of nu^Fm}
x \mapsto \hat{x} \cdot \nu^{F_{\mu}}_x ( x, z), 
\, \,  \, \, \,  \, 
x \mapsto | \nu^{F_{\mu}}_{\!  x {\scriptscriptstyle\parallel}} ( x, z)|, \, \,\,  \, \,  \, 
x \mapsto \nu^{F_{\mu}}_z ( x, z),
\end{align}
are constant in $( \partial^* F_{\mu} )_{(r, z)}$, that is $x\mapsto \nu^{F_\mu}_{\mathsf{c}}(x,z)$ is constant in $(\partial^*F_\mu)_{(r,z)}$.
\end{proposition}
\noindent
We observe that a weaker version of the above result in the Steiner setting was already known (see \cite[Remark 2.5]{barchiesicagnettifusco}). Let us now introduce some further notation that we will need in order to state the next theorem. Thanks to Proposition \ref{real magic property} we can define the Borel vector field $\bar{\nu}^{F_\mu}_{\mathsf{c}}:(0,\infty)\times\mathbb{R}^{k-2}\to \mathbb{R}^k$ as 
\begin{align}\label{def: bar(nu)}
\bar{\nu}^{F_\mu}_{\mathsf{c}}(r,z)=
\begin{cases}
\nu^{F_\mu}_{\mathsf{c}}(x,z)\qquad &\mbox{ if }(\partial^*F_\mu)_{(r,z)}\neq \emptyset,\textnormal{ and } x\in (\partial^*F_\mu)_{(r,z)},  \vspace*{0.1cm}\\
0\qquad &\mbox{ otherwise}.
\end{cases}
\end{align}
\noindent
Proposition \ref{real magic property} is the new ingredient for the characterisation of equality cases for the perimeter inequality under circular symmetrisation. In the following, we define the diffeomorphism $\Phi: (0, \infty) \times \mathbb{R}^{k-2}  \times \mathbb{S}^1 \to \mathbb{R}^2_0 \times \mathbb{R}^{k-2}$ as:
\begin{align*}
\Phi (r, z, \omega) := (r \omega, z) \quad &\text{ for every } (r,z, \omega) \in (0, \infty) \times \mathbb{R}^{k-2}  \times \mathbb{S}^1.
\end{align*}
Thus, more in general, for every Borel set $B\subset (0,\infty)\times \mathbb{R}^{k-2}$, we set
\begin{align*}
\Phi (B\times\mathbb{S}^1) := \left\{(x,z)\in \mathbb{R}^k:\, (|x|,z)\in B  \right\}.
\end{align*}
We can now state our main result.
\begin{theorem}\label{Eq. cases per. ineq.}
Let $\mu  : (0, \infty) \times \R^{k-2} \to [0, \infty)$ be a Lebesgue measurable function
satisfying \eqref{compatibility for mu}, let $U\subset (0, \infty) \times \R^{k-2}$ be an open set, and let $E\subset \mathbb{R}^k$ be a $\mu$-distributed set such that $E$ has finite perimeter in $\Phi(U\times\mathbb{S}^1)$. Then,  $F_\mu$ has finite perimeter in $\Phi(U\times\mathbb{S}^1)$ and
\begin{align}\label{eq: circular per ineq}
P (F_\mu; \Phi (B \times \mathbb{S}^{1})) \leq P (E; \Phi  (B \times \mathbb{S}^{1})),\quad \forall\,B \subset U \textnormal{ Borel}.
\end{align} 
Moreover, equality holds in \eqref{eq: circular per ineq} for some Borel set $B\subset U$ if and only if both the following two conditions are satisfied.
\begin{itemize}
\vspace{.1cm}
\item[a)]For $\mathcal{L}^{k-1}$-a.e. $(r,z)\in B$ we have that $(E)_{(r,z)}$ is $\mathcal{H}^1$-equivalent to a connected arc in $\R^2$.
\vspace{.1cm}
\item[b)] There exists $N\subset \partial^*E$ with $\mathcal{H}^{k-1}(N)=0$, with the property that \textbf{for every} $(r,z)\in B$ such that $( \partial^* E\setminus N )_{(r, z)} \neq \emptyset$, and $(\partial^*F_\mu)_{(r,z)} \neq \emptyset$, we have that
\begin{align}\label{eq: magic property for E}
\nu^{E}_{\mathsf{c}}(x,z)=\bar{\nu}^{F_\mu}_{\mathsf{c}}(r,z)\quad \forall x\in (\partial^*E\setminus N)_{(r,z)}.
\end{align}
\end{itemize}
\end{theorem}


\begin{remark}
By definition of $\nu^E_{\mathsf{c}}$, condition b) of the above result implies that \textbf{for every} $(r,z)\in B$ such that $( \partial^* E\setminus N )_{(r, z)} \neq \emptyset$, and $(\partial^*F_\mu)_{(r,z)} \neq \emptyset$ the functions
\begin{align*}
x \mapsto \hat{x} \cdot \nu^{E}_x ( x, z), 
\, \,  \, 
x \mapsto | \nu^{E}_{\!  x {\scriptscriptstyle\parallel}} ( x, z)|, \, \,\, 
x \mapsto \nu^{E}_z ( x, z),
\end{align*}
are constant in $(\partial^*E\setminus N)_{(r,z)}$.
\end{remark}
Roughly speaking, we can say that condition b) of Theorem \ref{Eq. cases per. ineq.} holds true if and only if the symmetric properties of $\nu^{F_\mu}$ described by Proposition \ref{real magic property} holds true also for $\nu^E$. Let us point out that in \cite[Theorem 1.4]{CagnettiPeruginiStoger} condition a) and a weaker version of condition b) were shown to be necessary condition for a set $E$ to be an extremal of \eqref{eq: circular per ineq}. In particular, condition b) of \cite[Theorem 1.4]{CagnettiPeruginiStoger} (see also condition b) of \cite[Theorem 1.1]{barchiesicagnettifusco}) was only discussed for $\mathcal{L}^{k-1}$-a.e. $(r,z)\in B$, and no information was given on the $\mathcal{L}^{k-1}$-negligible subset of $B$  where coarea formula cannot be used. In order to clarify the meaning of condition b) of Theorem \ref{Eq. cases per. ineq.}, let us give some examples.

\begin{figure}[!htb]
\centering
\def\svgwidth{13cm}
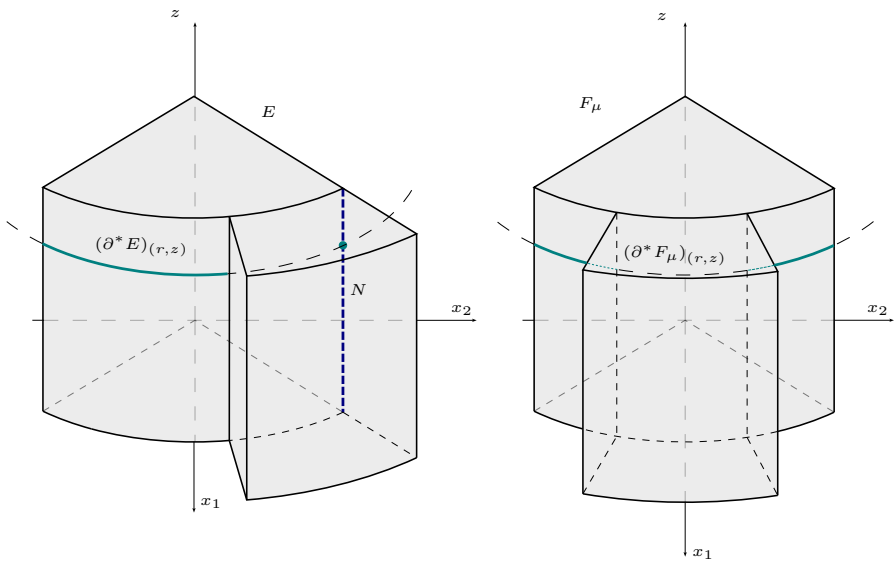
\caption{A pictorial representation of a $\mu$-distributed set $E\subset \mathbb{R}^3$ that satisfies both conditions a) and b) of Theorem \ref{Eq. cases per. ineq.}, thus being an equality case for \eqref{eq: circular per ineq}.}
\label{fig: eq cases}
\end{figure}

\begin{figure}[!htb]
\centering
\def\svgwidth{13cm}
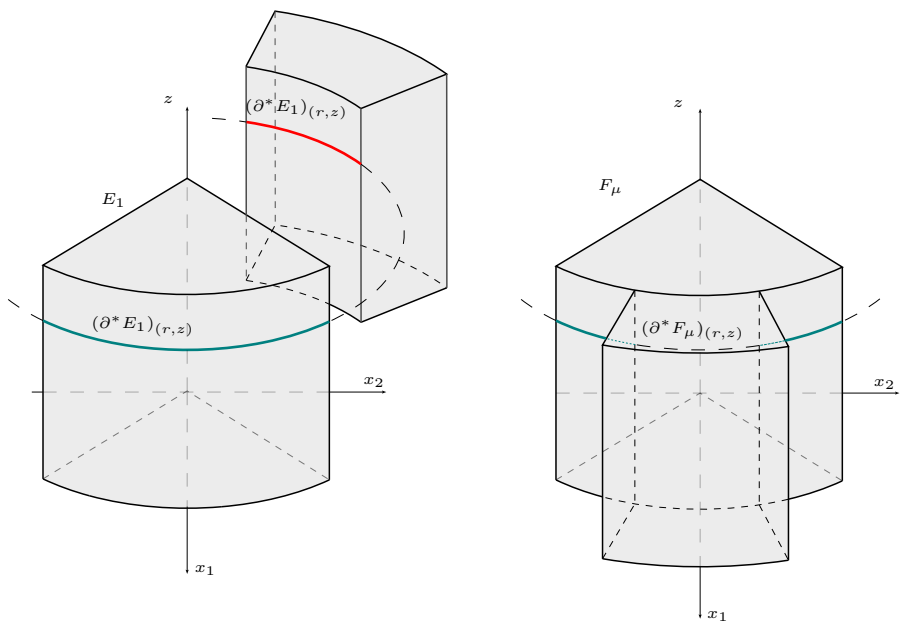
\caption{A pictorial representation of a $\mu$-distributed set $E_1\subset\mathbb{R}^3$ that does not satisfy condition b) of Theorem \ref{Eq. cases per. ineq.}.}
\label{fig: NO eq cases}
\end{figure}

\begin{example}[Case of equality]\label{ex: cases of equality}
Let us explain with an example in $\mathbb{R}^3$ the meaning of condition b) of Theorem \ref{Eq. cases per. ineq.}. In Figure \ref{fig: eq cases} we marked in green the sets $\left((\partial^*E_{(r,z)}),z\right)$, and $\left((\partial^*F_\mu)_{(r,z)},z\right)$ for some $(r,z)\in (0,\infty)\times\mathbb{R}$ (in the picture, with a little abuse of notation, we simply call them $(\partial^*E)_{(r,z)}$, and $(\partial^*F_\mu)_{(r,z)}$, respectively). In the picture in the left one can notice the $\mathcal{H}^2$-negligible set $N$, which is represented by the blue dashed vertical line. Let us point out that the isolated green dot appearing in the left picture is indeed part of $\left((\partial^*E_{(r,z)}),z\right)$, and it coincides with $\left((\partial^*E_{(r,z)}),z\right)\cap N$. It can be shown that $\nu^{E}_{\mathsf{c}}$ evaluated at that isolated point differs from $\nu^{E}_{\mathsf{c}}$ evaluated at any other point of $\left((\partial^*E_{(r,z)}),z\right)\setminus N$. Nonetheless, $\nu^{E}_{\mathsf{c}}$ is constant when restricted to $\left((\partial^*E_{(r,z)}),z\right)\setminus N$ and it coincides with $\nu^{F_\mu}_{\mathsf{c}}(x,z)$ restricted to $\left((\partial^*F_\mu)_{(r,z)},z\right)$, namely $\nu^E_{\mathsf{c}}(x,z)=\bar{\nu}^{F_\mu}_{\mathsf{c}}(r,z)$ for all $x \in (\partial^*E\setminus N)_{(r,z)}$. Thus, condition b) of Theorem \ref{Eq. cases per. ineq.} holds true. 
\end{example}

\begin{example}[Non equality case]
In Figure \ref{fig: NO eq cases} we show an example of a $\mu$-distributed set $E_1\subset \mathbb{R}^3$ that does not satisfy condition b) of Theorem \ref{Eq. cases per. ineq.}. Indeed, it can be shown that $\nu^{E_1}_{\mathsf{c}}(x,z)$ changes depending on weather $(x,z)$ belongs to the green or to the red part of $((\partial^*E_1)_{(r,z)},z)$. Note that this phenomenon cannot be avoided by removing an $\mathcal{H}^2$-negligible set from $\partial^*E_1$. Thus, condition b) is not satisfied and therefore $E_1$ is not an extremal of \eqref{eq: circular per ineq}. Let us stress that, despite the set $E_1$ does not satisfy condition b) of Theorem \ref{Eq. cases per. ineq.}, it does satisfy all the necessary conditions in order to be a case of equality for \eqref{eq: circular per ineq} that are listed in \cite[Theorem 1.4]{CagnettiPeruginiStoger}.
\end{example}

\noindent
Theorem \ref{Eq. cases per. ineq.} is a refinement of \cite[Theorem 1.4]{CagnettiPeruginiStoger}, where the inequality \eqref{eq: circular per ineq} was already stated without an explicit proof.  Let us stress that, apart from some technical intermediate results, the arguments we use to prove Theorem \ref{Eq. cases per. ineq.} differ from the standard ones used while proving perimeter inequalities under symmetrisation (see once more \cite[Theorem 1.1]{ChlebikCianchiFuscoAnnals05}, and  \cite[Theorem 1.1]{CagnettiPeruginiStoger}), and deeply rely on the new information given by Proposition \ref{real magic property} about the symmetral set $F_\mu$. Indeed, as a consequence of that, our proof of \eqref{eq: circular per ineq} is much more direct, and leads quite simply to the characterisation of the equality cases.

\medskip
Finally, we are able to show that analogous results hold true for the Steiner symmetrisation (see Theorem \ref{Steiner Eq. cases per. ineq.}). In fact, we believe that our short proof of \eqref{eq: circular per ineq} and the techniques we used to show Theorem \ref{Eq. cases per. ineq.} can be adapted to other symmetrisation procedures, and that they can be helpful in simplifying the study of \emph{rigidity} of perimeter inequality under symmetrisation.
\subsection*{Structure of the paper} The paper is divided as follows. In Section \ref{preliminaries} we recall some basic notions of geometric measure theory and functions of bounded variation. In Section \ref{Sec: properties of Fmu and mu} for the reader convenience we start off by stating once more the precise notation we will use throughout the paper, and then we focus on proving Proposition \ref{real magic property} and other technical results we will need later on. In Section \ref{sec: characterisation of equality cases} we present the proof of Theorem \ref{Eq. cases per. ineq.}. Lastly, in Section \ref{sec: Steiner setting} we state, without proofs, the Steiner counterpart of the results we obtained for the circular symmetrisation, thus including a Steiner version of both Proposition \ref{real magic property}, and Theorem \ref{Eq. cases per. ineq.} (see Proposition \ref{Steiner real magic property}, and Theorem \ref{Steiner Eq. cases per. ineq.}, respectively).

\subsection*{Acknowledgements}
The author would like to thank Filippo Cagnetti for his valuable comments and for inspiring conversations about the subject. 

\section{Fundamentals of geometric measure theory}
\label{preliminaries}
The aim of this section is to introduce some basic concepts of Geometric Measure Theory that will be largely used in the article. 
For more details the reader can have a look in the monographs 
\cite{AFP, GMSbook1, maggiBOOK, Simon83}. For $(x,z)\in\mathbb{R}^k$ and $\nu\in \mathbb{S}^{k-1}$, we will denote by $H_{(x,z),\nu}^+$ and $H_{(x,z),\nu}^-$ the closed half-spaces
whose boundaries are orthogonal to $\nu$:
\begin{align*}
H_{(x,z),\nu}^+:=\Big\{(\bar{x},\bar{z})\in\R^k:\,(x-\bar{x},z-\bar{z})\cdot\nu\ge 0\Big\},\quad
H_{(x,z),\nu}^-:=\Big\{(\bar{x},\bar{z})\in\R^k:\,(x-\bar{x},z-\bar{z})\cdot\nu\le 0\Big\}.
\end{align*}
In the following, given a measurable set $E\subset \mathbb{R}^k$ we will denote by $\chi_E$ its characteristic function, while the $k$-dimensional ball of $\mathbb{R}^k$ of radius $r>0$ and center in $(x,z)$ is denoted with $B_r(x,z)$.

\subsection{Density points} 
Let $E \subset \R^k$ be a Lebesgue measurable set and let $(x,z)\in\R^k$. 
The upper and lower $k$-dimensional densities of $E$ at $(x,z)$ are defined as
\begin{eqnarray*}
  \theta^*(E,(x,z)) :=\limsup_{\rho\to 0^+}\frac{\mathcal{H}^k(E\cap B_\rho(x,z))}{\omega_k\,\rho^k}\,,
  \qquad
  \theta_*(E,(x,z)) :=\liminf_{\rho\to 0^+}\frac{\mathcal{H}^k(E\cap B_\rho(x,z))}{\omega_k\,\rho^k}\,,
\end{eqnarray*}
respectively, where $\omega_k\, \rho^k =\mathcal{H}^k(B_\rho(x,z))$.
It turns out that $(x,z) \mapsto  \theta^*(E,(x,z))$ and $(x,z) \mapsto  \theta_*(E,(x,z))$
are Borel functions that agree $\mathcal{H}^k$-a.e. on $\R^k$. 
Therefore, the $k$-dimensional density of $E$ at $(x,z)$
\[
\theta(E,(x,z)) := \lim_{\rho\to 0^+}\frac{\mathcal{H}^k(E\cap B_\rho(x,z))}{\omega_k\,\rho^k}\,,
\]
is defined for $\mathcal{H}^k$-a.e. $(x,z)\in\R^k$, and $(x,z) \mapsto  \theta (E,(x,z))$
is a Borel function on $\R^k$.
Given $t \in [0,1]$, we set
$$
E^{(t)} :=\{(x,z)\in\R^k:\theta(E,(x,z))=t\}.
$$
The set $\partial^{\mathrm{e}} E :=\R^n\setminus(E^{(0)}\cup E^{(1)})$ is called the \textit{essential boundary} of $E$.

\subsection{Functions of bounded variation} 
Let $f:(0,\infty)\times \mathbb{R}^{k-2}  \to\R$ be a Lebesgue measurable function, 
and let $\Omega\subset (0,\infty)\times \mathbb{R}^{k-2}$ be open, such that $f\in L^1(\Omega)$. Then we say that $f$ is of bounded variation in $\Omega$, and we write $f\in BV(\Omega)$ if and only if 
\begin{align}\label{def: total variation of Df, con f in BV}
\sup\Big\{\int_\Omega\,f(r,z)\,\diver\,T(r,z)\,dr\,dz:\,T\in C^1_c(\Omega;\R^{k-1})\,,|T|\le 1\Big\}<\infty,
\end{align}
where $C^1_c(\Omega;\R^{k-1})$ is the set of $C^1$ functions from $\Omega$ to $\R^{k-1}$ with compact support. More in general, we say that $f\in BV_{\textnormal{loc}}(\Omega)$ if $f\in BV(\Omega')$ for every open set $\Omega'$ compactly contained in $\Omega$. If $f\in BV_{\textnormal{loc}}(\Omega)$ the distributional derivative $Df$ of $f$ is representable as a $\mathbb{R}^{k-1}$-valued Radon measure defined on $\Omega$, and its total variation $|Df|$ is finite in $\Omega$, and its value $|Df|(\Omega)$ coincides with \eqref{def: total variation of Df, con f in BV}. Moreover, for every $T\in C^1_c(\Omega;\mathbb{R}^{k-1})$ we have
\begin{align*}
\int_{\Omega}f(r,z)\textnormal{div}\,T(r,z)\,dr\,dz = -\int_{\Omega}T(r,z)\cdot dDf(r,z).
\end{align*} 
One can write the  Radon--Nykodim decomposition of $Df$ with respect to $\mathcal{L}^{k-1}$
as $Df=D^af+D^sf$, where $D^sf$ and $\mathcal{L}^{k-1}$ are mutually singular, and where $D^af\ll\mathcal{L}^{k-1}$. We denote the density of $D^af$ with respect to $\mathcal{L}^{k-1}$ by $\nabla f$, 
so that $\nabla\,f\in L^1(\Omega;\R^{k-1})$ with $D^af=\nabla f\,d\mathcal{L}^{k-1}$. Moreover, for $\mathcal{L}^{k-1}$-a.e. $(r,z)\in \Omega$, $\nabla f(r,z)$ is the approximate differential of $f$ at $(r,z)$. 

\subsection{Sets of finite perimeter}\label{section sofp}

\noindent
Let $E\subset\R^k$ be a Lebesgue measurable set, and let  $O\subset \mathbb{R}^k$ be an open set. We say that  $E\subset \mathbb{R}^k$ is a set of finite perimeter in $O$ if and only if 
\begin{align}\label{def: locally finite perimeter of a set in R^k}
\sup\left\{\int_{\mathbb{R}^k}\chi_E(x,z)\textnormal{div}_{(x,z)}\,T(x,z)\,dx\,dz:\, T\in C^1_c(O;\mathbb{R}^k)  \right\}<\infty,
\end{align}
where by $\textnormal{div}_{(x,z)}$ we mean the classical divergence in $\mathbb{R}^k$ w.r.t. the variables $(x,z)$. If $E\subset \mathbb{R}^k$ is a set of finite perimeter in $O$, we denote with $P(E;O)$ its relative perimeter in $O$, where $P(E;O)$ coincides with the quantity in \eqref{def: locally finite perimeter of a set in R^k}. If $P(E):=P(E;\mathbb{R}^k)<\infty$ we say that $E$ is a set of finite perimeter, while more generally if $P(E;V)<\infty$ for every $V\subset\subset O$, we say that $E$ is a set of locally finite perimeter in $O$. If $E\subset \mathbb{R}^k$ is a set of finite perimeter and finite volume in $O$, then we have that $\chi_E\in BV(O)$, while in general if $E\subset \mathbb{R}^k$ is a set of finite perimeter in $O$ then $\chi_E\in BV_{\textnormal{loc}}(O)$. Moreover, if $E\subset \mathbb{R}^k$ is a set of locally finite perimeter in $O$ we define the \emph{reduced boundary} $\partial^*E\subset \mathbb{R}^k$ of $E$ as the set of those points such that
\begin{align*}
\nu^E(x,z):=\lim_{\rho\to 0^+}\frac{D\chi_E(B_\rho(x,z))}{|D\chi_E|(B_\rho(x,z))},
\end{align*}
exists and belongs to $\mathbb{S}^{k-1}$. The Borel function $\nu^E:\partial^*E\to \mathbb{S}^{k-1}$ 
is called the {\it (measure-theoretic) inner unit normal} to $E$. Given $E\subset \mathbb{R}^k$ set of locally finite perimeter in $O$, we have that $D\chi_E=\nu^E\mathcal{H}^{k-1}\mres (\partial^*E\cap O)$ and,
\begin{align*}
\int_{\mathbb{R}^k}\chi_E(x,z)\textnormal{div}_{(x,z)}\,T(x,z)\,dx\,dz=-\int_{\partial^*E\cap O}T(x,z)\cdot \nu^E(x,z)\,d\mathcal{H}^{k-1}(x,z),\quad \forall\,T\in C^1_c(O;\mathbb{R}^k).
\end{align*}
The relative perimeter of $E$ in $A\subset O$ is then defined by
$$
P(E;A):=|D\chi_E|(A)=\mathcal{H}^{k-1}(\partial^*E\cap A)
$$
for every Borel set $A \subset O$. If $E$ is a set of locally finite perimeter in $O$, it turns out that 
\begin{equation*}
  \label{inclusioni frontiere}
  (\partial^*E\cap O)  \subset (E^{(1/2)}\cap O) \subset (\partial^{\mathrm{e}} E\cap O)\,.
\end{equation*}
Moreover, {\it Federer's theorem} holds true (see \cite[Theorem 3.61]{AFP} and \cite[Theorem 16.2]{maggiBOOK}):
\[
\mathcal{H}^{n-1}((\partial^{\mathrm{e}} E\cap O)\setminus(\partial^*E\cap O))=0.
\]

\section{Properties of $F_\mu$ and $\mu$}\label{Sec: properties of Fmu and mu}
\noindent
We start this section stating two important results. The first one, is a special case of Coarea Formula (see
\cite[Proposition~6.1]{CagnettiPeruginiStoger}, and \cite[Theorem 18.8]{maggiBOOK}). In the following, given $O\subset \mathbb{R}^{k}$ open set, and given $E\subset\mathbb{R}^k$ set of locally finite perimeter in $O$, we denote with $L^1(\mathbb{R}^{k},\mathcal{H}^{k-1}\mres \partial^*E\cap O)$ the space of integrable functions from $\mathbb{R}^k$ to $\mathbb{R}$ w.r.t. the Radon measure $\mathcal{H}^{k-1}\mres \partial^*E\cap O$.
\begin{proposition} \label{coarea for sets} 
Let $\mu  : (0, \infty) \times \R^{k-2} \to [0, \infty)$ be a Lebesgue measurable function
satisfying \eqref{compatibility for mu}, let $U\subset (0, \infty) \times \R^{k-2}$ be an open set, and let $E\subset \mathbb{R}^k$ be a $\mu$-distributed set such that $E$ has finite perimeter in $\Phi(U\times\mathbb{S}^1)$. Let $g:\R^{k} \rightarrow [-\infty,\infty]$ be a Borel function, such that either $g\geq 0$ on $\partial^*E\cap \Phi(U\times\mathbb{S}^1)$, or $g\in L^1(\mathbb{R}^{k},\mathcal{H}^{k-1}\mres \partial^*E\cap \Phi(U\times\mathbb{S}^1))$.
Then, 
\begin{align*}
\int_{\partial^* E\cap \Phi(U\times\mathbb{S}^1)} g(x, z) | \nu^{E}_{\!  x {\scriptscriptstyle\parallel}} (x,z)| \, d\mathcal{H}^{k-1} (x, z) 
= \int_{U} dr \, dz \int_{(\partial^* E)_{(r, z) }} g(x, z) \, d\mathcal{H}^{0}(x).
\end{align*}
\end{proposition}
\noindent
Next result is about circular one-dimensional slices of sets of finite perimeter (see \cite[Theorem 6.2]{CagnettiPeruginiStoger}), and it can be seen as the circular counterpart of a classic result by Vol'pert (see \cite{Volpert}, and \cite[Theorem D]{cianchifusco2}).
\begin{proposition}[Vol'pert] \label{volpert theorem sets}
Let $\mu  : (0, \infty) \times \R^{k-2} \to [0, \infty)$ be a Lebesgue measurable function
satisfying \eqref{compatibility for mu}, let $U\subset (0, \infty) \times \R^{k-2}$ be an open set, and let $E\subset \mathbb{R}^k$ be a $\mu$-distributed set such that $E$ has finite perimeter in $\Phi(U\times\mathbb{S}^1)$.
Then, there exists a Borel set $G_{E} \subset (\{ \mu > 0 \}\cap U)$ with $\mathcal{L}^{k-1} ((\{ \mu > 0\}\cap U) \setminus G_{E}) = 0$ such that the following properties hold true:
\begin{itemize}
\item[(i)] for every $(r, z) \in G_{E}$:
\vspace{.2cm}
\begin{itemize}
\item[(ia)] $E_{(r,z)}$ is a set of finite perimeter in $\partial B (r)$;
\vspace{.3cm}
\item [(ib)] $\partial^* \left( E_{(r,z)} \right) = ( \partial^{*} E)_{(r,z)}$;
\end{itemize}
\end{itemize}
\end{proposition}

\noindent
We are now going to show some useful properties of the symmetric set $F_{\mu}$. Such properties (see Proposition \ref{real magic property}) are the new fundamental ingredient with which proving the perimeter inequality under circular symmetrisation, and then characterize the cases of equality. In the following, for every $\gamma \in [-2\pi, 2\pi]$ we define $R_\gamma$ as
the the counterclockwise rotation of an angle $\gamma$ in the plane $(x_1, x_2)$. Lastly, for every $(r,z)\in (0,\infty)\times \mathbb{R}^{k-2}$ we set
\begin{align}\label{eq: weird 1-d circle}
\partial B^2_r(0,z):=\left\{(x,z)\in\mathbb{R}^2_0\times\mathbb{R}^{k-2}:\, |x|=r \right\}.
\end{align}
Roughly speaking, $\partial B^2_r(0,z)$ stands for the 1-dimensional circle in $\mathbb{R}^2_0\times \{z \}\subset \mathbb{R}^k$ centered in $(0,z)$ and having radius $r$.

\begin{lemma}\label{lem: decreasing densities}
Let $\mu  : (0, \infty) \times \mathbb{R}^{k-2} \to [0, \infty)$ be a Lebesgue measurable function satisfying \eqref{compatibility for mu}. Let $(r, z) \in (0, \infty) \times \mathbb{R}^{k-2}$, and set $x_r := (r, 0)$. 
Then, the functions 
\[
\gamma \mapsto \theta_* (F_{\mu},  ( R_\gamma x_r , z))
\quad 
\text{ and } 
\quad
\gamma \mapsto \theta^* (F_{\mu},  ( R_\gamma x_r ,z ))
\]
are even in $[-\pi, \pi]$ and non increasing in $[0, \pi]$.
\end{lemma}

\begin{proof}
The fact that \( \gamma \mapsto \theta_* (F_{\mu},  ( R_\gamma x_r , z)) \)
and \( \gamma \mapsto \theta^* (F_{\mu},  ( R_\gamma x_r , z)) \)
are even in $[-\pi, \pi]$ follows directly from Remark~\ref{rem used for monotonicity}. 
We now divide the rest of the proof into two steps.

\vspace{.2cm}

\noindent{\textbf{Step 1:}} We show that, if 
$0 \leq \gamma_1 < \gamma_2 \leq \pi$, and $\rho > 0$ is so small that  
\begin{equation} \label{ssss}
B_\rho ((R_{\gamma_1} x_r, z)) \cap B_\rho ( (R_{\gamma_2} x_r, z) ) = \emptyset,
\end{equation}
 then for every $(\lambda, \overline{z}) 
\in (0, \infty) \times \mathbb{R}^{k-2}$ one has
\begin{align} 
&\mathcal{H}^1 \Big( F_{\mu} \cap \partial B^2_\lambda (0,\overline{z} )
\cap B_\rho ( (R_{\gamma_2} x_r,z)) \cap T_{\gamma_2} A \Big) \nonumber\\
&\leq \mathcal{H}^1 \Big( F_{\mu} \cap \partial B^2_\lambda (0, \overline{z})
\cap B_\rho ( (R_{\gamma_1} x_r, z)) \cap T_{\gamma_1} A \Big), \label{ferp}
\end{align}
for every $A \subset \mathbb{R}^{k}$ where, for every $\gamma \in [-\pi, \pi]$, we set
\[
T_\gamma (x, z) := (R_\gamma x, z).
\]

If \(F_{\mu} \cap B_\rho ( (R_{\gamma_2} x_r, z)) \cap \partial B^2_\lambda (0, \overline{z}) = \emptyset \), 
 the left hand side of \eqref{ferp} equals $0$ 
and therefore the inequality is satisfied. 
Instead, suppose that 
\[
F_{\mu} \cap B_\rho ( (R_{\gamma_2} x_r, z)) \cap \partial B^2_\lambda (0, \overline{z} ) \neq \emptyset. 
\]
Then, from \eqref{ssss} and Remark~\ref{rem used for monotonicity} we have
\[
F_{\mu} \cap B_\rho ( (R_{\gamma_1} x_r, z)) \cap \partial B^2_\lambda (0,\overline{z} ) 
= B_\rho ( (R_{\gamma_1} x_r,z)) \cap \partial B^2_\lambda (0, \overline{z} ).
\]
Therefore, 
\begin{align*}
&\mathcal{H}^1 \Big( F_{\mu} \cap \partial B^2_\lambda (0, \overline{z})
\cap B_\rho ( (R_{\gamma_2} x_r,z)) \cap T_{\gamma_2} A \Big) \\
&\leq \mathcal{H}^1 \Big( \partial B^2_\lambda (0, \overline{z} )
\cap B_\rho ( (R_{\gamma_2} x_r, z)) \cap T_{\gamma_2} A \Big) \\
&= \mathcal{H}^1 \Big( \partial B^2_\lambda (0, \overline{z} )
\cap B_\rho ( (R_{\gamma_1} x_r, z)) \cap T_{\gamma_1} A \Big) \\
&= \mathcal{H}^1 \Big( F_{\mu} \cap \partial B^2_\lambda (0, \overline{z})
\cap B_\rho ( (R_{\gamma_1} x_r, z)) \cap T_{\gamma_1} A \Big),
\end{align*}
which gives \eqref{ferp}.

\vspace{.2cm}

\noindent{\textbf{Step 2:}} We will show that 
if $0 \leq \gamma_1 < \gamma_2 \leq \pi$, then

\[
\theta_* (F_{\mu},  ( R_{\gamma_2} x_r,z))
\leq 
\theta_* (F_{\mu},  ( R_{\gamma_1} x_r, z)),
\]
and
\[
\theta^* (F_{\mu},  ( R_{\gamma_2} x_r, z))
\leq 
\theta^* (F_{\mu},  ( R_{\gamma_1} x_r, z)).
\]
Let $\rho > 0$ be such that \eqref{ssss} is satisfied. Then,
\begin{align*}
&\mathcal{H}^{k} (B_\rho  ( R_{\gamma_1} x_r, z)) \cap F_{\mu})
= \int_{B_\rho ( ( R_{\gamma_1} x_r, z))} \chi_{F_{\mu}} (\overline{x},\overline{z}) \, d \mathcal{H}^{k} (\overline{x},\overline{z}) \\
& = \int_{\mathbb{R}^{k-2}} \int_{r - \rho}^{r + \rho}
\mathcal{H}^1 (F_{\mu} \cap B_\rho ( ( R_{\gamma_1} x_r, z)) \cap \partial B^2_\lambda (0, \overline{z} ))
\, d \lambda \, d \overline{z} \\
& \geq \int_{\mathbb{R}^{k-2} } \int_{r - \rho}^{r + \rho}
\mathcal{H}^1 (F_{\mu} \cap B_\rho ( ( R_{\gamma_2} x_r, z)) \cap \partial B^2_\lambda (0, \overline{z}))
\, d \lambda \, d \overline{z}   \\
&= \mathcal{H}^{k} (B_\rho ( ( R_{\gamma_2} x_r, z)) \cap F_{\mu}),
\end{align*}
where the inequality follows from \eqref{ferp} with $A = \mathbb{R}^{k}$. 
Thus, 
\[
\frac{\mathcal{H}^{k} (B_\rho  ( R_{\gamma_1} x_r, z)) \cap F_{\mu})}{\omega_{k} \rho^{k}}
\geq 
\frac{\mathcal{H}^{k} (B_\rho  ( R_{\gamma_2} x_r, z)) \cap F_{\mu})}{\omega_{k} \rho^{k}}.
\]
Passing to the liminf and the limsup as $\rho \to 0^+$, the conclusion follows.
\end{proof}

\begin{proposition} \label{magic property}
Let $\mu  : (0, \infty) \times \mathbb{R}^{k-2}\to [0, \infty)$ be a Lebesgue measurable function
satisfying \eqref{compatibility for mu} such that $F_\mu$ is a set of locally finite perimeter.
Suppose that $(x, z) \in \partial^* F_{\mu}$, and let $r \in (0, \infty)$ and $\beta \in (-\pi, \pi]$ be such that
\( x = r (\cos \beta, \sin \beta) \).
Then,
\begin{align}
\nu^{F_{\mu}} (R_{\gamma} x, z) = 
\left( R_{\gamma} \nu^{F_{\mu}}_x ( x, z), 
\nu^{F_{\mu}}_z ( x, z) \right), \label{formula symmetric normals in slices}
\end{align}
for every $\gamma \in \left[ \min\{ - \beta, 0 \}, \max\{ - \beta, 0 \} \right]$ such that
$(R_{\gamma} x, z) \in ( \partial^* F_{\mu} )_{(r, z)}$. 
\end{proposition}
Roughly speaking, what the above result says is that, given $(x,z)\in \partial^*F_\mu$ as in the statement, if there exists any other point $(\bar{x},z) \in \partial^*F_\mu$ satisfying the following properties, namely $\bar{x}\in (\partial^*F_\mu)_{(|x|,z)}$, $\arccos (\hat{\bar{x}}\cdot e_1)\leq |\beta|$, and $x_2\,\bar{x}_2\geq 0$, then there exists an angle $\gamma\in [\min\{-\beta,0  \},\max\{-\beta,0 \}]$ such that $\bar{x}= R_\gamma x $ and the corresponding $\nu^{F_\mu}(\bar{x},z)$ can be written as
$$
\nu^{F_\mu}(\bar{x},z)=(R_\gamma\nu^{F_\mu}_x(x,z),\nu^{F_\mu}_z(x,z)).
$$

\begin{proof}[Proof of Proposition~\ref{magic property}]
In the following, we set $x_r = (r, 0)$.   
If $\beta = 0$ there is nothing to prove, so we can assume $\beta \neq 0$.
We will only consider the case $\beta > 0$, since for $\beta < 0$
the proof is analogous.
Also, since $R_0 x = x$ and the statement is true for $x$, 
we only need to consider the case $\gamma \neq 0$. Let $\gamma \in [ - \beta, 0 )$, and $\rho > 0$ be such that $r-\rho>0$, and 
\begin{align*}
\emptyset 
&= B_\rho ((x, z)) \cap B_\rho ( (R_{\gamma} x, z) ) 
= B_\rho ((R_{\beta} x_r, z)) \cap B_\rho ( (R_{\gamma + \beta} x_r, z) ).
\end{align*}
In the following, to ease the notation, let us set
\[
\nu = \nu^{F_{\mu}} ( x, z) \quad \text{ and } \quad
\nu_\gamma = ( R_{\gamma} \nu^{F_{\mu}}_x ( x, z), 
\nu^{F_{\mu}}_z ( x, z) ).
\]
We have
\begin{align*}
&\mathcal{H}^{k} 
\Big( H^{+}_{(R_{\gamma} x, z), \nu_{\gamma} } \cap F_{\mu} \cap B_\rho ((R_{\gamma} x, z)) \Big) 
=\mathcal{H}^{k} 
\Big( H^{+}_{(R_{\gamma + \beta} x_r, z), \nu_{\gamma} } \cap F_{\mu} \cap B_\rho ((R_{\gamma + \beta} x_r, z)) \Big) \\
& = \int_{\mathbb{R}^{k-2} } \int_{r - \rho}^{r + \rho}
\mathcal{H}^1 (F_{\mu} \cap B_\rho ( ( R_{\gamma + \beta}  x_r, z)) \cap H^{+}_{(R_{\gamma + \beta} x_r, z), \nu_{\gamma} } \cap \partial B^2_\lambda (0, \overline{z} ))
\, d \lambda \, d \overline{z} \\
&= \int_{\mathbb{R}^{k-2}} \int_{r - \rho}^{r + \rho}
\mathcal{H}^1 (F_{\mu} \cap B_\rho ( ( R_{\gamma + \beta}  x_r, z)) \cap 
T_\gamma ( H^{+}_{(R_\beta x_r, z), \nu}  ) \cap \partial B^2_\lambda (0, \overline{z} ))
\, d \lambda \, d \overline{z} \\
& \geq \int_{\mathbb{R}^{k-2}} \int_{r - \rho}^{r + \rho}
\mathcal{H}^1 (F_{\mu} \cap B_\rho ( ( R_{\beta}  x_r, z)) \cap H^{+}_{(R_{\beta} x_r, z), \nu } \cap \partial B^2_\lambda (0, \overline{z} ))
\, d \lambda \, d \overline{z} \\
&= \int_{\mathbb{R}^{k-2}} \int_{r - \rho}^{r + \rho}
\mathcal{H}^1 (F_{\mu} \cap B_\rho ( ( x,z)) \cap H^{+}_{(x, z), \nu} \cap \partial B^2_\lambda (0, \overline{z} ))
\, d \lambda \, d \overline{z} \\
&= \mathcal{H}^{k} 
\Big( H^{+}_{(x, z), \nu} \cap F_{\mu} \cap B_\rho ((x, z)) \Big), 
\end{align*}
where in the inequality we used \eqref{ferp} with $A = H^{+}_{(x, z), \nu}$, 
and the fact that $\gamma < 0$.
From the last chain of inequalities we obtain
\begin{align*}
& \frac{\mathcal{H}^{k} 
\Big( H^{+}_{(R_{\gamma} x, z), \nu_{\gamma} }  \cap B_\rho ((R_{\gamma} x, z)) \Big)}
{\omega_{k} \rho^{k}} 
\geq  \frac{\mathcal{H}^{k} 
\Big( H^{+}_{(R_{\gamma} x, z), \nu_{\gamma} } \cap F_{\mu} \cap B_\rho ((R_{\gamma} x, z)) \Big)}
{\omega_{k} \rho^{k}}
\\
&\geq  \frac{\mathcal{H}^{k} 
\Big( H^{+}_{(x, z), \nu} \cap F_{\mu} \cap B_\rho ((x, z)) \Big)}
{\omega_{k} \rho^{k}}.
\end{align*}
Passing to the limit as $\rho \to 0^+$, we have 
\begin{align*}
&\frac{1}{2} 
= \lim_{\rho \to 0^+} 
\frac{\mathcal{H}^{k} 
\Big( H^{+}_{(R_{\gamma} x, z), \nu_{\gamma} }  \cap B_\rho ((R_{\gamma} x, z)) \Big)}
{\omega_{k} \rho^{k}} 
\geq \limsup_{\rho \to 0^+}  \frac{\mathcal{H}^{k} 
\Big( H^{+}_{(R_{\gamma} x, z), \nu_{\gamma} } \cap F_{\mu} \cap B_\rho ((R_{\gamma} x, z)) \Big)}
{\omega_{k} \rho^{k}} \\
&\geq \liminf_{\rho \to 0^+}  \frac{\mathcal{H}^{k} 
\Big( H^{+}_{(R_{\gamma} x, z), \nu_{\gamma} } \cap F_{\mu} \cap B_\rho ((R_{\gamma} x, z)) \Big)}
{\omega_{k} \rho^{k}} 
\geq  \lim_{\rho \to 0^+} \frac{\mathcal{H}^{k} 
\Big( H^{+}_{(x, z), \nu} \cap F_{\mu} \cap B_\rho ((x, z)) \Big)}
{\omega_{k} \rho^{k}} = \frac{1}{2}, 
\end{align*}
where the last equality follows from the fact that $\nu$ is the inner unit normal to
$\partial^* F_{\mu}$ at $(x, z)$.
Therefore, 
\[
\frac{1}{2} = \lim_{\rho \to 0^+}  \frac{\mathcal{H}^{k} 
\Big( H^{+}_{(R_{\gamma} x, z), \nu_{\gamma} } \cap F_{\mu} \cap B_\rho ((R_{\gamma} x, z)) \Big)}
{\omega_{k} \rho^{k}}.
\]
Since by assumption $R_{\gamma} x \in \partial^* F_{\mu}$, 
it has to be 
\[
\nu^{F_{\mu}} (R_{\gamma} x, z) = 
\left( R_{\gamma} \nu^{F_{\mu}}_x ( x, z), 
\nu^{F_{\mu}}_z( x, z) \right), 
\]
and this allows us to conclude.
\end{proof}
\noindent
Now we state a useful remark. For a similar result in the context of Steiner symmetrisation
see \cite[Remark~2.5]{barchiesicagnettifusco}.

\begin{remark}\label{rem: nu ^F_mu is symmetric wrt to x2=0}
Let us notice that, by symmetry of the set $F_\mu$ w.r.t. the hyperplane $\{ x_2=0 \}\subset \mathbb{R}^k$, the following property holds true. Given any $\nu\in\mathbb{R}^{2}$ we denote with $\textnormal{Ref}(\nu)\in \mathbb{R}^2$ the reflection of $\nu$  with respect to $\{x_2=0\}\subset \mathbb{R}^2$, namely $\textnormal{Ref}(\nu)=(\nu_1,-\nu_2)$.  Then, for every $(x,z)\in \partial^*F_\mu$ we have that $(\textnormal{Ref}(x),z)\in \partial^*F_\mu$ and 
$$
\nu^{F_\mu}(\textnormal{Ref}(x),z)= \left(\textnormal{Ref}(\nu^{F_\mu}_x(x,z)),\nu^{F_\mu}_z(x,z)  \right).
$$ 
\end{remark}
\noindent
We are now ready to prove Proposition \ref{real magic property}.


\begin{proof}[Proof of Proposition \ref{real magic property}]
Let  $(r, z) \in (0, \infty) \times \R^{k-2} $
such that the slice $( \partial^* F_{\mu} )_{(r,z)} \neq \emptyset$. We divide the slice in two parts, namely
$$
( \partial^* F_{\mu} )_{(r,z)}=( \partial^* F_{\mu} )_{(r,z)}^+ \cup ( \partial^* F_{\mu} )_{(r,z)}^-
$$
where we set $( \partial^* F_{\mu} )_{(r,z)}^{+}=( \partial^* F_{\mu} )_{(r,z)}\cap \{ x_2 \geq 0 \}$, and  $( \partial^* F_{\mu} )_{(r,z)}^{-}=( \partial^* F_{\mu} )_{(r,z)}\cap \{ x_2 < 0 \}$. We now divide the proof in steps, depending on how many points are contained in the slice.\\
\textbf{Step 1a.} Let us suppose that $\mathcal{H}^{0}(( \partial^* F_{\mu} )_{(r,z)}^{+})=1$. Let $x\in \mathbb{R}^2_0$ such that $\{x\}=( \partial^* F_{\mu} )_{(r,z)}^{+}$, and suppose in addition that $x_2=0$. Then, by symmetry properties of $F_\mu$, the point $x$ is the only point in the entire slice $( \partial^* F_{\mu} )_{(r,z)}$, and so we conclude.\\
\textbf{Step 1b.}  Let us suppose that $\mathcal{H}^{0}(( \partial^* F_{\mu} )_{(r,z)}^{+})=1$. Let $x\in \mathbb{R}^2_0$ such that $\{x\}=( \partial^* F_{\mu} )_{(r,z)}^{+}$, and suppose in addition that $x_2>0$. Then, by symmetry properties of $F_\mu$, the points $x$, and $\textnormal{Ref}(x)$ namely the reflection of  $x$ w.r.t $\{ x_2=0 \}$ (see Remark \ref{rem: nu ^F_mu is symmetric wrt to x2=0}), are the only points in the entire slice $( \partial^* F_{\mu} )_{(r,z)}$. Applying Remark \ref{rem: nu ^F_mu is symmetric wrt to x2=0} we get that the two vectors $\nu^{F_\mu}(x,z)$, and $\nu^{F_\mu}(\textnormal{Ref}(x),z)$ are symmetric to each other w.r.t. $\{ x_2=0 \}$, and so by a direct computation we show that the three functions in \eqref{eq: the components of nu^Fm} are constant in the slice $( \partial^* F_{\mu} )_{(r,z)}$. This concludes the second part of the first step.\\
\textbf{Step 2.} Let us suppose that $\mathcal{H}^{0}(( \partial^* F_{\mu} )_{(r,z)}^{+})>1$. Let $x\in ( \partial^* F_{\mu} )_{(r,z)}^{+}$ and let $\beta\in (0,\pi]$ be such that $x=r(\cos\beta,\sin\beta)$. Thanks to Proposition \ref{magic property} we get that
\begin{align*}
\nu^{F_{\mu}} (R_{\gamma} x, z) = 
\left( R_{\gamma} \nu^{F_{\mu}}_x ( x, z), 
\nu^{F_{\mu}}_z ( x, z) \right), 
\end{align*}
for every $\gamma \in \left[ -\beta,0 \right]$ such that
$(R_{\gamma} x, z) \in ( \partial^* F_{\mu} )_{(r, z)}^+$. 
As a consequence of the fact that the above relation holds true for every $x\in( \partial^* F_{\mu} )_{(r, z)}^+$, we get that the three functions in \eqref{eq: the components of nu^Fm} are constant in $( \partial^* F_{\mu} )_{(r, z)}^+$. By symmetry of $F_\mu$ w.r.t. $\{ x_2=0 \}$, the same conclusion holds true when restricting the three functions in \eqref{eq: the components of nu^Fm} to $( \partial^* F_{\mu} )_{(r, z)}^-$. Finally, the fact that the constant values of those three functions does not change when passing from $( \partial^* F_{\mu} )_{(r, z)}^+$ to $( \partial^* F_{\mu} )_{(r, z)}^-$ is a consequence of Remark \ref{rem: nu ^F_mu is symmetric wrt to x2=0}. This concludes the proof of the second step. Putting together the informations obtained in all these steps, we conclude. 
\end{proof}
\noindent
We now focus our attention on the properties of the function $\mu$. Parts of the following results were already stated without an explicit proof in \cite[Section 6]{CagnettiPeruginiStoger}. For completeness, and future references we decide to provide here a detailed proof. In the following, we denote by $C^0_c (\Omega ; \R^{k-1})$ the class of all continuous functions from $\Omega$ to $\R^{k-1}$, while  with $C^0_b (\Omega; \R^{k-1} )$ we denote the set of continuous and bounded function from $\Omega$ to $\R^{k-1}$.

\medskip
\noindent
In the following, given $\mu  : (0, \infty) \times \R^{k-2} \to [0, \infty)$ a Lebesgue measurable function
satisfying \eqref{compatibility for mu}, we denote by $\xi  : (0, \infty) \times \R^{k-2} \to [0, \infty)$ the function defined as 
\begin{align}\label{def:xi}
\xi(r,z):=\mu(r,z)/r\quad \textnormal{for }\mathcal{L}^{k-1}\textnormal{-a.e. }(r,z)\in (0, \infty) \times \R^{k-2}.
\end{align}

\begin{lemma}\label{lem: mu e xi in BV loc}
Let $\mu  : (0, \infty) \times \R^{k-2} \to [0, \infty)$ be a Lebesgue measurable function
satisfying \eqref{compatibility for mu}, let $U\subset (0, \infty) \times \R^{k-2}$ be an open set, and let $E\subset \mathbb{R}^k$ be a $\mu$-distributed set such that $E$ has finite perimeter in $\Phi(U\times\mathbb{S}^1)$.  Then, both the function $\mu$, and the function  $\xi$ defined in \eqref{def:xi} are in $BV_{\textnormal{loc}} (U)$. In addition, $|D_z \mu|$, and $|rD_r \xi|$ are finite Radon measures 
on $U$, and  
for every Borel set $B \subset U$ we have
\begin{align}
&\int_{B} \varphi (r, z) \, d D_{z_i} \mu (r, z)
= \int_{\partial^*E \cap  \Phi (B \times \mathbb{S}^{1}) }  
\varphi (|x| , z) \, \nu_{z_i}^{E} (x, z) \, d \mathcal{H}^{k-1} (x,z), \label{Dz mu for E} \\
&\int_{B} \varphi (r, z) r d D_r \xi (r,z) 
\hspace{-.05cm}= 
\hspace{-.1cm} \int_{\partial^* E \cap  \Phi (B \times \mathbb{S}^{1})}  
\varphi (|x| ,z) \, \hat{x} \hspace{-.05cm} \cdot \hspace{-.05cm} \nu^{E}_x (x, z) \, d \mathcal{H}^{k-1} (x,z), \label{rDr xi for E}
\end{align}
for every $i\in\{1,\dots,k-2 \}$, and for every bounded Borel function $\varphi: B \to \mathbb{R}$. Moreover, let $\sigma_{\mu}$ be the $\mathbb{R}^k$-valued Radon measure on $U$ defined as
\begin{align}\label{ def of sigma}
\sigma_{\mu}(B):=\int_{B}d(rD_r \xi,\,2\mathcal{L}^{k-1}\mres(\{ \mu>0 \}\cap U),\,D_z \mu )(r, z),\quad \forall\,B\subset U \textnormal{ Borel}.
\end{align}
Then, for every Borel set $B \subset U$ we get 
\begin{align}\label{eq: the real calc lemmata weaker}
\int_B\varphi(r,z)\cdot d\sigma_{\mu}(r,z)\leq\int_{\partial^*E\cap\Phi(B\times\mathbb{S}^1)}\varphi(|x|,z)\cdot\nu^E_{\mathsf{c}}(x,z)\,d\mathcal{H}^{k-1}(x,z),
\end{align}
for every bounded Borel function $\varphi: B \to \mathbb{R}^k$ with non-negative second component, where $\nu^E_{\mathsf{c}}$ was defined in \eqref{eq: the circular version of nu^E}. In particular, equality sign holds true in \eqref{eq: the real calc lemmata weaker} if and only if $(E)_{(r,z)}$ is $\mathcal{H}^1$-equivalent to a connected arc for $\mathcal{L}^{k-1}$-a.e. $(r,z)\in B$.
\end{lemma}
\begin{proof}[Proof of Lemma \ref{lem: mu e xi in BV loc}]
We divide the proof in several steps.\\
\textbf{Step 1.} Let us prove that $\mu\in BV_{\textnormal{loc}}(U)$. Let us start by proving that $\mu\in L^1_{\textnormal{loc}}(U)$. Let $V\subset\subset U$, then
\begin{align*}
\|\mu\|_{L^1(V)}=\int_{V}\mu(r,z)\,dr\, dz=\int_{V}dr\, dz\int_{E_{(r,z)}}1\,d\mathcal{H}^1(x)=\int_{E\cap \Phi(V\times\mathbb{S}^1)}1\,d\mathcal{H}^{k}<\infty.
\end{align*}
This proves that $\mu\in L^1_{\textnormal{loc}}(U)$. Similarly, we get that $\xi \in L^1_{\textnormal{loc}}(U)$. In order to conclude this first step we need to show that for every $V\subset\subset U$ open set, we have
\begin{align}\label{eq: total variation of mu}
\sup\left\{\int_{V} \mu(r,z)\diver\,T(r,z)\,dr\, dz:\,T\in C^1_c(V;\mathbb{R}^{k-1}),\,|T|\leq 1 \right\}<\infty.
\end{align}
Let $i\in\{1,\dots,k-2 \}$, and let $\varphi\in C^1_c(V)$ with $|\varphi|\leq 1$. Then, 
\begin{align*}
&\int_{V}\mu(r,z)\frac{\partial\varphi}{\partial z_i}(r,z)\,dr\,dz=\int_V dr\,dz\int_{E_{(r,z)}}\frac{\partial\varphi}{\partial z_i}(|x|,z)\,d\mathcal{H}^{1}(x)\\
&=\int_{\Phi(V\times\mathbb{S}^1)}\chi_E(x,z)\frac{\partial\varphi}{\partial z_i}(|x|,z)\,d\mathcal{H}^{k}(x,z)=-\int_{\partial^*E\cap \Phi(V\times\mathbb{S}^1)}\varphi(|x|,z)\nu^E_{z_i}(x,z)\,d\mathcal{H}^{k-1}(x,z)\\
&\leq P(E;\Phi(V\times\mathbb{S}^1))<\infty.
\end{align*}

Let us now recall that, for any $\varphi\in C^1(V)$
\begin{align*}
\text{div}_{(x,z)} \big( \varphi (|x|, z) \hat{x} \big)
= \frac{\partial \varphi}{\partial r} (| x | , z) + \frac{1}{|x|} \varphi (|x|, z),
\end{align*}
where by $\text{div}_{(x, z)}$ we denoted the divergence in $\mathbb{R}^{k}$ with respect to the variables $(x, z)$; we do that to distinguish when we consider the divergence in $(0,\infty)\times\mathbb{R}^{k-2}$ w.r.t. the variables $r$, and $z$. Then, for any $\varphi\in C^1_c(V)$ with $|\varphi|\leq 1$ we get
\begin{align*}
&\int_{V}\mu(r,z)\frac{\partial\varphi}{\partial r}(r,z)\,dr\,dz =  \int_{\Phi (V \times \mathbb{S}^{1})} \chi_{E} (x, z) 
\frac{\partial \varphi}{\partial r} (|x| , z)\, dx\, dz   \\
&=  \int_{\Phi (V \times \mathbb{S}^{1})} \chi_{E} (x, z)  
\left(  \text{div}_{(x, z)} \Big( \varphi (|x|, z) \hat{x} \Big)
-  \frac{1}{|x|} \varphi (|x|, z) \right) \, dx \, dz \\
&= - \int_{\partial^*E\cap\Phi (V \times \mathbb{S}^{1})} 
\varphi (|x| , z) \, \hat{x} \cdot \nu^{E}(x, z)\, d\mathcal{H}^{k-1}(x,z)
- \int_{\Phi (V \times \mathbb{S}^{1})} \chi_{E} (x, z) 
\frac{1}{|x|} \varphi (|x|, z)  \, dx \, dz  \\
&= -\int_{\partial^* E \cap \Phi (V \times \mathbb{S}^{1}) }  
\varphi (|x| , y, t) \, \hat{x} \cdot \, \nu^{E}_x (x, z) \, d \mathcal{H}^{k-1} (x, z) 
-\int_{V} \xi (r, y, t) \varphi (r, y, t) \, dr \, dz,\\
&\leq P(E;\Phi(V\times\mathbb{S}^1)) + \|\xi \|_{L^1(V)}<\infty,
\end{align*}
where for the last inequality we used that $\xi\in L^1_{\textnormal{loc}}(U)$, and $V\subset\subset U$. Putting together the above calculations we get that \eqref{eq: total variation of mu} holds true, and this proves that $\mu\in BV_{\textnormal{loc}}(U)$. Since the maps $(r, z) \mapsto 1/r$ and $(r, z) \mapsto \mu (r, z)$ belong to $BV (V)$, thanks to \cite[Example~3.97]{AFP} we have that 
$\xi (r, z) = \mu (r, z)/r \in BV (V)$ for every $V\subset\subset U$ open set, and so $\xi (r, z)\in BV_{\textnormal{loc}} (U)$. In particular, 
\begin{align}\label{eq: chian rule fo mu e xi}
D_r \mu = D_r (r \xi) = r D_r \xi + \xi \, dr \, dz.
\end{align}
This concludes the first step.\\ 

\noindent
\textbf{Step 2a.} Let us prove that relations \eqref{Dz mu for E}, and \eqref{rDr xi for E} holds true for every $\varphi\in C^1_c (U)$. Let $\varphi\in C^1_c(U)$ be a test function, and let $V\subset\subset U$ be an open set such that $\textnormal{supp}(\varphi)\subset V$. Then, by properties of $BV_{\textnormal{loc}}$ functions, together with the calculation we made in the first step, we have
\begin{align*}
&-\int_{U}\varphi(r,z)\,dD_{z_i}\mu(r,z)=\int_{U}\mu(r,z)\frac{\partial\varphi}{\partial_{z_i}}(r,z)\,dr\,dz\\
&=\int_{V}\mu(r,z)\frac{\partial\varphi}{\partial_{z_i}}(r,z)\,dr\,dz=-\int_{\partial^*E\cap \Phi(V\times\mathbb{S}^1)}\varphi(|x|,z)\nu^E_{z_i}(x,z)\,d\mathcal{H}^{k-1}(x,z)\\
=&-\int_{\partial^*E\cap \Phi(U\times\mathbb{S}^1)} \varphi(|x|,z)\nu^E_{z_i}(x,z)\,d\mathcal{H}^{k-1}(x,z).
\end{align*}
Thus, for all $i\in\{1,\dots,\k-2 \}$, and for all $\varphi\in C^1_c(U)$ we have,
\begin{align}\label{eq: intermediate calculation lemmas dz}
\int_{U}\varphi(r,z)\,dD_{z_i}\mu(r,z)=\int_{\partial^*E}\varphi(|x|,z)\nu^E_{z_i}(x,z)\,d\mathcal{H}^{k-1}(x,z),
\end{align}
which proves that \eqref{Dz mu for E} holds true for every $\varphi\in C^1_c(U)$. Let us now prove that \eqref{rDr xi for E} holds true for every $\varphi\in C^1_c (U)$. Let $\varphi\in C^1_c(U)$ be a test function, and let $V\subset\subset U$ be an open set such that $\textnormal{supp}(\varphi)\subset V$. Then, analogously to what we proved above, by properties of $BV_{\textnormal{loc}}$ functions, together with the calculation we made in the first step, we have
\begin{align*}
&-\int_{U} \varphi (r, z) \, d D_r \mu (r,z) 
= \int_{U} \mu(r,z) \frac{\partial \varphi}{\partial r}(r,z)\, dr\,dz =\int_V \mu(r,z) \frac{\partial \varphi}{\partial r}(r,z)\, dr\,dz\\
&=-\int_{\partial^* E \cap \Phi (V \times \mathbb{S}^{1}) }  
\varphi (|x| , y, t) \, \hat{x} \cdot \, \nu^{E}_x (x, z) \, d \mathcal{H}^{k-1} (x, z) 
-\int_{V} \xi (r, y, t) \varphi (r, y, t) \, dr \, dz,\\
&=-\int_{\partial^* E\cap \Phi(U\times\mathbb{S}^1)}  
\varphi (|x| , y, t) \, \hat{x} \cdot \, \nu^{E}_x (x, z) \, d \mathcal{H}^{k-1} (x, z) 
-\int_{U} \xi (r, y, t) \varphi (r, y, t) \, dr \, dz,
\end{align*}
from which we get
\begin{align}\label{eq: intermediate calculation lemmas dr 1}
\int_{U} \varphi (r, z) \, d D_r \mu (r,z) 
&= \int_{\partial^* E\cap \Phi(U\times\mathbb{S}^1)}  
\varphi (|x| , z) \, \hat{x} \cdot \, \nu^{E}_x (x,z) \, d \mathcal{H}^{k-1} (x, z)  \\
&\hspace{.2cm}+
 \int_{U}   \xi (r, z) \varphi (r, z) \, dr \, dz \nonumber.
\end{align}
Comparing \eqref{eq: intermediate calculation lemmas dr 1} with \eqref{eq: chian rule fo mu e xi}, we get that \eqref{rDr xi for E} holds true for every $\varphi\in C^1_c(U)$. Before concluding this first step, let us observe as a consequence of the previous calculations, we have that
\begin{align}
&|D_{z_i}\mu|(U)\leq P(E; \Phi(U\times\mathbb{S}^1)),\qquad \textnormal{for }i=1,\dots,k-2,   \label{eq: |Dzi mu|< infty}\\ \medskip
&|rD_r\xi|(U)\leq P(E; \Phi(U\times\mathbb{S}^1)).\label{eq: |rDr xi|< infty}
\end{align}
This proves that $|D_z\mu|$ and $|rD_r\xi|$ are finite Radon measures on $U$, and we conclude the first step.\\
\noindent
\textbf{Step 2b.} We are now ready to prove \eqref{Dz mu for E}, and \eqref{rDr xi for E} whenever $B\subset \subset U$. We will only show \eqref{Dz mu for E}, since the proof of \eqref{rDr xi for E} is similar. Let $i\in \{1,\dots,k-2\}$, let $B\subset\subset U$ be a Borel set, let $\varphi:B\to \mathbb{R}$ be a bounded Borel function, and let $V\subset\subset U$ open set such that $B\subset V$. We call $\bar{\varphi}:V\to\mathbb{R}$ the Borel function that coincides with $\varphi$ in $B$, and it is zero in $V\setminus B$. Since every function in $C^0_b (V)$ can be approximated uniformly on compact subsets of $V$ by functions in $C^1_c (V)$, and since $D_{z_i} \mu$ is a bounded Radon measure on $V$, we have that
\eqref{Dz mu for E} holds true for every function in $C^0_b (V)$. Let $\lambda$ be the 
bounded Radon measure on $V$ defined by
\begin{align}\label{def: auxiliary measure lambda}
\lambda (B) := | D_{z_i} \mu| (B) + \mathcal{H}^{k-1}  \left( \partial^* E \cap \left( \Phi (B \times \mathbb{S}^{1}) \right) \right) 
\end{align}
for every Borel set $B \subset V$. By Lusin Theorem, for every $h \in \mathbb{N}$ there exists $\varphi_h \in C^0_b (V)$ such that $\| \varphi_h \|_{L^{\infty} (V)} \leq \| \bar{\varphi} \|_{L^{\infty} (V)}$ and
\begin{align*}
\lambda \left( \{ (r, z) \in V : \bar{\varphi} (r,z) \neq \varphi_h (r, z) \}  \right) < \frac{1}{h}.
\end{align*}
For each $h \in \mathbb{N}$ we can apply 
\eqref{Dz mu for E} to $\varphi_h$, obtaining 
\begin{equation*} 
\int_{V} \varphi_h (r, z) \, d D_{z_i}\mu (r, z)
= \int_{\partial^*E \cap \left( \Phi (V \times \mathbb{S}^{1}) \right)}  
\varphi_h (|x| , z) \, \nu_{z_i}^{E} (x, z) \, d \mathcal{H}^{k-1} (x, z).
\end{equation*}
Using this identity, we have
\begin{align*}
&\left| \int_{B} \varphi (r, z) \, d D_{z_i} \mu (r,z) - \int_{\partial^* E \cap  \Phi (B \times \mathbb{S}^{1}) }  
\varphi (|x| , z) \, \nu_{z_i}^{E} (x, z) \, d \mathcal{H}^n (x, z) \right| \\
&=\left| \int_{V} \bar{\varphi} (r, z) \, d D_{z_i} \mu (r,z) - \int_{\partial^* E \cap  \Phi (V \times \mathbb{S}^{1}) }  
\bar{\varphi} (|x| , z) \, \nu_{z_i}^{E} (x, z) \, d \mathcal{H}^n (x, z) \right|\\ 
&\leq \left| \int_{V} \big( \bar{\varphi} (r,z) - \varphi_h (r, z) \big) \, d D_{z_i} \mu (r, z) \right| \\
&+ \left| \int_{V} \varphi_h (r, z) \, d D_{z_i} \mu (r, z) - \int_{\partial^* E \cap  \Phi (V \times \mathbb{S}^{1})} \varphi_h (|x| ,z) \, \nu_{z_i}^{E} (x,z) \, d\mathcal{H}^{k-1} (x, z) \right| \\
&+ \left| \int_{\partial^* E \cap  \Phi (V \times \mathbb{S}^{1})}  \left( \bar{\varphi} (|x| ,z) - \varphi_h (r,z) \right) \, \nu_{z_i}^{E} (x,z) \, d \mathcal{H}^{k-1} (x, z) \right| \\
&= \left| \int_{V} \big( \bar{\varphi} (r, z) - \varphi_h (r, z) \big) \, d D_{z_i} \mu (r,z) \right| \\
&+ \left|  \int_{\partial^* E \cap \Phi (V \times \mathbb{S}^{1})}  
\big( \bar{\varphi} (r, z) - \varphi_h (r, z) \big) \, \nu_{z_i}^{E} (x, z) \, d \mathcal{H}^{k-1} (x, z) \right| \\
&\leq \int_{V} \left| \bar{\varphi} (r,z) - \varphi_h (r, z) \right| \, d \left| D_{z_i} \mu \right| (r,z) \\
&+ \int_{\partial^* E \cap  \Phi (V \times \mathbb{S}^{1})}  \left| \bar{\varphi} (r,z) - \varphi_h (r, z) \right| \, d \mathcal{H}^{k-1} (x,z) 
\leq \frac{4}{h} \| \bar{\varphi} \|_{L^{\infty} (V)}.
\end{align*}
Passing to the limit as $h \to \infty$ we obtain \eqref{Dz mu for E} whenever $B\subset \subset U$. This concludes step 2b.\\
\textbf{Step 2c.} We finally prove \eqref{Dz mu for E}, and \eqref{rDr xi for E}. As done in step 2b, we will only show \eqref{Dz mu for E}. Fix $i\in \{1,\dots,k-2 \}$ and consider the Radon measure $\lambda$ on $U$ defined as in \eqref{def: auxiliary measure lambda}. Let $B\subset U$ be a Borel set, and let $(B_h)_{h\in\mathbb{N}}\subset B$ be a sequence of compact sets, with the property that $\lambda(B\setminus B_h)<\epsilon_h$, where $(\epsilon_h)_{h\in\mathbb{N}}\subset [0,1]$ and $\lim_{h\to \infty}\epsilon_h=0$. Let $\varphi:B\to \mathbb{R}$ be a bounded Borel function, and let us set $\varphi_h(r,z)= \chi_{B_h}(r,z)\varphi(r,z)$ for every $(r,z)\in B$, for every $h\in \mathbb{N}$. By construction, up to pass to a subsequence, we have that $\lim_{h\to \infty}\varphi_h(r,z)=\varphi(r,z)$ for $\lambda$-a.e. $(r,z)\in B$. Thus,
\begin{align*}
&\left|\int_{B}\varphi(r,z)\,dD_{z_i}\mu(r,z)-\int_{\partial^*E\cap \Phi(B\times\mathbb{S}^1)}\varphi(|x|,z)\nu^E_{z_i}(x,z)\,d\mathcal{H}^{k-1}(x,z)  \right|\\
&=\left|\int_{B}(\varphi(r,z)-\varphi_h(r,z))\,dD_{z_i}\mu(r,z)-\int_{\partial^*E\cap \Phi(B\times\mathbb{S}^1)}(\varphi(|x|,z)-\varphi_h(|x|,z))\nu^E_{z_i}(x,z)\,d\mathcal{H}^{k-1}(x,z)  \right|\\
&\leq \sup_{(r,z)\in B}\varphi(r,z)\,\lambda (B\setminus B_h) \leq \epsilon_h\sup_{(r,z)\in B}\varphi(r,z). 
\end{align*}
Passing to the limit in the above relation as $h\to \infty$ we prove \eqref{Dz mu for E}. Formula \eqref{rDr xi for E} can be obtained in similar way using the approximation argument we just presented.\\
\noindent
\textbf{Step 3.} Let us prove \eqref{eq: the real calc lemmata weaker}. Let $B\subset U$ be a Borel set, and let $g: B\to [0,\infty]$ be a Borel function. Let us denote with $\mathrm{Pr}(\partial^*E)$ the projection in $U$ of the set $\partial^*E\cap \Phi(U\times\mathbb{S}^1)$, namely
$$
\mathrm{Pr}(\partial^*E):=\left\{(r,z)\in U:\, (\partial^*E)_{(r,z)}\neq \emptyset  \right\}.
$$

By construction, it can be shown that $\mathcal{L}^{k-1}(\mathrm{Pr}(\partial^*E)\setminus (\{\mu>0 \}\cap U))=0$, while by Proposition \ref{volpert theorem sets} we have that $\mathcal{L}^{k-1}((\{ \mu>0 \}\cap U)\setminus \mathrm{Pr}(\partial^*E))=0$.
Thus, by the Coarea formula \eqref{coarea for sets} we get
\begin{align*}
&\int_{B\cap \{ \mu>0 \}\cap U} 2g (r, z) \, dr\,dz=\int_{B\cap \mathrm{Pr}(\partial^*E)} 2g (r, z) \, dr\,dz\leq \int_{B\cap \mathrm{Pr}(\partial^*E)} g(r, z)\int_{(\partial^*E)_{(r,z)}}1\,d\mathcal{H}^{0}(x) \, dr\,dz\\&=\int_{\partial^*E \cap \Phi ((B\cap \mathrm{Pr}(\partial^*E)) \times \mathbb{S}^{1}) }  
g (|x| , z) \,|\nu^{E}_{\!  x {\scriptscriptstyle\parallel}}(x,z)| \, d \mathcal{H}^{k-1} (x,z)\\
&= \int_{\partial^*E \cap \Phi (B \times \mathbb{S}^{1}) }  
g (|x| , z) \,|\nu^{E}_{\!  x {\scriptscriptstyle\parallel}}(x,z)| \, d \mathcal{H}^{k-1} (x,z),
\end{align*}
where for the inequality sign we used Proposition \ref{volpert theorem sets}, and the properties of the set $\mathrm{Pr}(\partial^*E)$ to infer that $\mathcal{H}^0((\partial^*E)_{(r,z)})\geq 2$ for $\mathcal{L}^{k-1}$-a.e. $(r,z)\in \mathrm{Pr}(\partial^*E)$. The above relation, together with \eqref{Dz mu for E}, and \eqref{rDr xi for E} proves \eqref{eq: the real calc lemmata weaker}. This concludes the third step and the proof of the lemma.
\end{proof}

\begin{remark}\label{rem: weak calculation lemmata specifically for Fmu}
Under the assumptions of the above lemma, let $B\subset U$ be a Borel set, and let $E\subset \mathbb{R}^{k}$ be a $\mu$-distributed set of finite perimeter in $\Phi(U\times\mathbb{S}^1)$ such that $(E)_{(r,z)}$ is $\mathcal{H}^1$-equivalent to a connected arc for $\mathcal{L}^{k-1}$-a.e. $(r,z)\in B$. Then, we get  that \eqref{eq: the real calc lemmata weaker} holds true with equality and in addition, as a consequence of Proposition \ref{coarea for sets}, we can drop the assumption of the non-negativity of the second component of the vector field appearing the formula, namely 
\begin{align}\label{eq: the real calc lemmata for Fmu}
\int_B\varphi(r,z)\cdot d\sigma_{\mu}(r,z)=\int_{\partial^*E\cap\Phi(B\times\mathbb{S}^1)}\varphi(|x|,z)\cdot\nu^E_{\mathsf{c}}(x,z)\,d\mathcal{H}^{k-1}(x,z),
\end{align}
for every bounded Borel function $\varphi: B \to \mathbb{R}^k$.
\end{remark}

\noindent
Let us mention that the idea of considering the Radon measure $\sigma_{\mu}$ was inspired by \cite[Section 4.1.5]{GMSbook1}. The next result can be seen as a refinement of \cite[Proposition 6.8]{CagnettiPeruginiStoger}.

\begin{lemma}\label{lem: P(F[v])=|sigma|}
Let $\mu  : (0, \infty) \times \R^{k-2} \to [0, \infty)$ be a Lebesgue measurable function
satisfying \eqref{compatibility for mu}, let $U\subset (0, \infty) \times \R^{k-2}$ be an open set, and let $E\subset \mathbb{R}^k$ be a $\mu$-distributed set such that $E$ has finite perimeter in $\Phi(U\times\mathbb{S}^1)$. Then, the set $F_\mu\subset \mathbb{R}^{k}$ defined in \eqref{def: F_mu} is a set of finite perimeter in $\Phi(U\times\mathbb{S}^1)$. Moreover,
\begin{align}\label{eq: |sigma|=Per (Fmu)}
|\sigma_{\mu}|(B)=P(F_\mu;\Phi(B\times\mathbb{S}^1)),\quad \forall\,B\subset U \textnormal{ Borel},
\end{align}
where $\sigma_{\mu}$ is the Radon measure defined in \eqref{ def of sigma}.
\end{lemma}
\begin{proof}
We divide the proof in several steps. We start by proving that the set  $F_\mu\subset\mathbb{R}^k$ is of finite perimeter in $\Phi(U\times\mathbb{S}^1)$. The argument we are going to use is standard, but for the seek of completeness and for future references we decided to include it (see \cite[Proposition 4.3]{CagnettiPeruginiStoger} for the same argument but in the spherical symmetrisation setting). Let $\Omega\subset \subset U$ be an open set. By Lemma \ref{lem: mu e xi in BV loc} $\xi\in BV(\Omega)$. Thus, by standard approximation techniques, let $(\xi_j)_{j\in\mathbb{N}}\subset C^1_c(\Omega;\mathbb{R}^k)$ be a sequence of non negative functions such that $\xi_j\to\xi$ for $\mathcal{L}^{k-1}$-a.e. $(r,z)\in\Omega$, and $|\nabla \xi_j|\mathcal{L}^{k-1}\stackrel{*}{\rightharpoonup} |D\xi|$, where the function $\xi$ was defined in \eqref{def:xi}, and with the symbol $\stackrel{*}{\rightharpoonup}$ we denote the weak star convergence of Radon measures. In the following, denoting with $\mu_j(r,z)=r\xi_j(r,z) $, we call $F_{\mu_j}\subset \mathbb{R}^k$ the set defined as in \eqref{def: F_mu} w.r.t. the function $\mu_j$.\\
\textbf{Step 0.} In this step we present some circular notation that we will need for the following calculations. Let $\varphi\in C^1_c(\Phi(\Omega\times\mathbb{S}^1),\mathbb{R}^k)$ with $|\varphi|\leq 1$. A direct calculation shows that 
\begin{align}
\text{div}_{(x,z)}\varphi(x,z)&=\text{div}_{(x)}\varphi_x(x,z)+\text{div}_{(z)}\varphi_z(x,z)\nonumber\\
&=\text{div}_{(x )\parallel}\varphi_{x {\scriptscriptstyle\parallel}}(x,z)+\nabla_x\varphi_x(x,z)[\hat{x}]\cdot\hat{x}+\frac{\varphi_x(x,z)\cdot\hat{x}}{|x|}+\text{div}_{(z)}\varphi_z(x,z),
\end{align}
where $\text{div}_{(x)}$, and $\text{div}_{(z)}$ stand for the classical divergence in $\mathbb{R}^2$ w.r.t. the variables $x_1$, and $x_2$, and the classical divergence in $\mathbb{R}^{k-2}$ w.r.t. the variables $z_1,\dots,z_{k-2}$, respectively, $\text{div}_{(x )\parallel}\varphi_{x {\scriptscriptstyle\parallel}}(x,z)$ stands for the tangential divergence in $\mathbb{R}^2$ of $\varphi_{x {\scriptscriptstyle\parallel}}(\cdot,z)$ at $(x,z)$ in $\partial B(|x|)$, and finally $\nabla_x$ is the classical gradient in $\mathbb{R}^2$ w.r.t. the variables $x_1$, and $x_2$. Thus,
\begin{align}\label{eq: div Fi = I+II+III}
\int_{\Phi(\Omega\times\mathbb{S}^1)}\chi_{F_{\mu_j}}(x,z)\text{div}_{(x,z)}\varphi(x,z)\, dx\, dz= \textbf{I}+\textbf{II}+\textbf{III},
\end{align}
where we set
\begin{align*}
\textbf{I}&:=\int_{\Phi(\Omega\times\mathbb{S}^1)}\chi_{F_{\mu_j}}(x,z)\text{div}_{(x)\parallel}\varphi_{x{\scriptscriptstyle\parallel}}(x,z)\,dx\,dz;\\
\textbf{II}&:=\int_{\Phi(\Omega\times\mathbb{S}^1)}\chi_{F_{\mu_j}}(x,z)\nabla_x\varphi_x(x,z)[\hat{x}]\cdot\hat{x}\,dx\,dz + \int_{\Phi(\Omega\times\mathbb{S}^1)}\chi_{F_{\mu_j}}(x,z)\frac{\varphi_x(x,z)\cdot\hat{x}}{|x|}\,dx\, dz;\\
\textbf{III}&:= \int_{\Phi(\Omega\times\mathbb{S}^1)}\chi_{F_{\mu_j}}(x,z)\text{div}_{(z)}\varphi_{z}(x,z)\,dx\,dz.
\end{align*}
\textbf{Step 1.} In this step we study the quantity identified with \textbf{I}. Let us observe that, by construction of $F_{\mu_j}$, the slice $(F_{\mu_j})_{(r,z)}$ is a connected arc in $\partial B(r)$ for $\mathcal{L}^{k-1}$-a.e. $(r,z)\in \Omega$. By the theory of sets of finite perimeter in $\partial B(r)$ (see for instance \cite[Section 3.2]{CagnettiPeruginiStoger}), and Proposition \ref{volpert theorem sets} we get that
$$2=\mathcal{H}^0(\partial^*((F_{\mu_j})_{(r,z)}))\leq \mathcal{H}^0(\partial^*((E)_{(r,z)}))=\mathcal{H}^0((\partial^*E)_{(r,z)})\quad \textnormal{for }\mathcal{L}^{k-1}\textnormal{-a.e. }(r,z)\in \{\mu_j>0 \}\subset \Omega,
$$
while $\mathcal{H}^0(\partial^*((F_{\mu_j})_{(r,z)}))=0$ for $\mathcal{L}^{k-1}$-a.e. $(r,z)\in \Omega\setminus \{\mu_j>0 \}$.
Thus, applying the structure theorem for sets of finite perimeter on $\partial B(r)$ (see once more \cite[Section 3.2]{CagnettiPeruginiStoger}) we get
\begin{align}\label{eq: Fmu f.p. step 1}
&\int_{\Phi(\Omega\times\mathbb{S}^1)}\chi_{F_{\mu_j}}(x,z)\text{div}_{(x)\parallel}\varphi_{x{\scriptscriptstyle\parallel}}(x,z)\,dx\,dz= \int_{\Omega}dr\,dz\int_{(F_{\mu_j})_{(r,z)}}\text{div}_{(x)\parallel}\varphi_{x{\scriptscriptstyle\parallel}}(x,z)\,d\mathcal{H}^1(x)\nonumber\\
&\leq \int_{\Omega}\mathcal{H}^0((\partial^*E)_{(r,z)})\,dr\,dz =\int_{\partial^*E\cap \Phi(\Omega\times\mathbb{S}^1)}|\nu^E_{x{\scriptscriptstyle\parallel}}(x,z)|\,d\mathcal{H}^{k-1}(x,z)\leq P(E;\Phi(U\times\mathbb{S}^1)),
\end{align}
where for the last equality sign we used Proposition \ref{coarea for sets}. This concludes the first step.\\
\textbf{Step 2.} In this step we study the quantity identified with \textbf{II}. Let us introduce the following quantity
\begin{align*}
V_j(r,z):=\int_{(F_{\mu_j})_{(r,z)}}\varphi_x(x,z)\cdot \hat{x}\,d\mathcal{H}^1(x)=r\int_{-\xi_j(r,z)/2}^{\xi_j(r,z)/2}\varphi_x(r\omega(\theta),z)\cdot\omega(\theta)\,d\theta \quad \forall\, (r,z)\in \Omega,
\end{align*}
where $\xi(r,z)=\mu_j(r,z)/r$, and $\omega(\theta)=(\cos(\theta),\sin(\theta))\in \mathbb{S}^1$. By regularity properties of $\mu_j$, and of $\varphi_x$ the above quantity is differentiable in the $r$ variable, and a direct computation shows that
\begin{align*}
&\frac{\partial}{\partial r}V_j(r,z)=\int_{-\xi_j(r,z)/2}^{\xi_j(r,z)/2}\varphi_x(r\omega(\theta),z)\cdot\omega(\theta)\,d\theta\\
&+\frac{1}{2}r\frac{\partial}{\partial r}\xi_j(r,z)\left( \varphi_x(r\omega(\xi_j(r,z)/2),z)\cdot \omega(\xi_j(r,z)/2) +  \varphi_x(r\omega(-\xi_j(r,z)/2),z)\cdot \omega(-\xi_j(r,z)/2) \right) \\
&+ r\int_{-\xi_j(r,z)/2}^{\xi_j(r,z)/2}\nabla_x\varphi_x(r\omega(\theta),z)[\omega(\theta)]\cdot\omega(\theta)\,d\theta.
\end{align*}
In order to keep the notation a bit more compact, let us set
\begin{align*}
A_{x}(r,z)&= \varphi_x(r\omega(\xi_j(r,z)/2),z)\cdot \omega(\xi_j(r,z)/2),\\
B_{x}(r,z)&=  \varphi_x(r\omega(-\xi_j(r,z)/2),z)\cdot \omega(-\xi_j(r,z)/2).
\end{align*}
Let us observe that, by construction, the function $V_j$ has compact support in $\Omega$. Thus, integrating both sides of the above relation over $\Omega$ and applying Fubini theorem together with the fundamental theorem of calculus we get
\begin{align*}
&0=\int_{\Omega}\frac{\partial}{\partial r}V_j(r,z)\,dr\,dz=\int_{\Omega}dr\,dz\int_{-\xi_j(r,z)/2}^{\xi_j(r,z)/2}\varphi_x(r\omega(\theta),z)\cdot\omega(\theta)\,d\theta\\
&+\frac{1}{2}\int_{\Omega}r\frac{\partial}{\partial r}\xi_j(r,z)A_x(r,z) \,dr\,dz\\
&+\frac{1}{2}\int_{\Omega}r\frac{\partial}{\partial r}\xi_j(r,z)B_x(r,z) \,dr\,dz\\
&+\int_{\Omega} r\int_{-\xi_j(r,z)/2}^{\xi_j(r,z)/2}\nabla_x\varphi_x(r\omega(\theta),z)[\omega(\theta)]\cdot\omega(\theta)\,d\theta\,dr\,dz.
\end{align*}

Thus, after a changing variables, we get 
\begin{align}\label{eq: Fmu f.p. step 2}
&\int_{\Phi(\Omega\times\mathbb{S}^1)}\chi_{F_{\mu_j}}(x,z)\frac{\varphi_x(x,z)\cdot \hat{x}}{|x|}\,dx\,dz + \int_{\Phi(\Omega\times\mathbb{S}^1)}\chi_{F_{\mu_j}}(x,z)\nabla_x\varphi_x(x,z)[\hat{x}]\cdot\hat{x}\,dx\,dz\nonumber\\
&=-\frac{1}{2}\int_{\Omega}r\frac{\partial}{\partial r}\xi_j(r,z) A_x(r,z)\,dr\,dz - \frac{1}{2}\int_{\Omega}r\frac{\partial}{\partial r}\xi_j(r,z)B_x(r,z) \,dr\,dz
\end{align}
This concludes the second step.\\
\textbf{Step 3.} In this step we study the quantity identified with \textbf{III}. Similarly to what we did in the previous step, we consider the following auxiliary quantity
\begin{align*}
Z_j^i(r,z):=\int_{(F_{\mu_j})_{(r,z)}}(\varphi_{z})_i(x,z)\,d\mathcal{H}^1(x)=r\int_{-\xi_j(r,z)/2}^{\xi_j(r,z)/2}(\varphi_z)_i(x,z)\,d\theta \quad \forall\, (r,z)\in \Omega,
\end{align*}
where by $(\varphi_{z})_i$ stands for the $i$-th component of the vector $\varphi_z$, with $i=1,\dots,k-2$. Let us set
$$
\nabla_z\xi_j(r,z)= \left(\frac{\partial}{\partial z_1}\xi_j(r,z),\dots, \frac{\partial}{\partial  z_{k-2}}\xi_j(r,z) \right).
$$
Following verbatim the argument used in the step 2, and calling

\begin{align*}
A_z(r,z)=  \varphi_z(r\omega(\xi_j(r,z)/2),z),\quad
B_z(r,z)=\varphi_z(r\omega(-\xi_j(r,z)/2),z),
\end{align*}
we get that
\begin{align}\label{eq: Fmu f.p. step 3}
&\int_{\Phi(\Omega\times\mathbb{S}^1)}\chi_{F_{\mu_j}}(x,z)\text{div}_{(z)}\varphi_z(x,z)\,dx\,dz\nonumber\\
&=-\frac{1}{2}\int_{\Omega}r\nabla_z\xi_j(r,z)\cdot \left( \varphi_z(r\omega(\xi_j(r,z)/2),z) + \varphi_z(r\omega(-\xi_j(r,z)/2),z) \right)\,dr\,dz\\
&=-\frac{1}{2}\int_{\Omega}r\nabla_z\xi_j(r,z)\cdot A_z(r,z)\,dr\,dz - \frac{1}{2}\int_{\Omega}r\nabla_z\xi_j(r,z)\cdot B_z(r,z)\,dr\,dz.\nonumber
\end{align}
This concludes the third step.\\
\textbf{Step 4.} In this step we finally prove that $F_\mu$ has finite perimeter in $\Phi(U\times\mathbb{S}^1)$. Indeed, thanks to the previous step, in particular plugging into \eqref{eq: div Fi = I+II+III} the relations obtained in \eqref{eq: Fmu f.p. step 1}, \eqref{eq: Fmu f.p. step 2}, and \eqref{eq: Fmu f.p. step 3} we get
\begin{align*}
&\int_{\Phi(\Omega\times\mathbb{S}^1)}\chi_{F_{\mu_j}}(x,z)\text{div}_{(x,z)}\varphi(x,z)\, dx\, dz\leq P(E;\Phi(U\times\mathbb{S}^1))\\ 
&-\frac{1}{2} \int_{\Omega}r\nabla \xi_j(r,z)\cdot(A_x(r,z),A_z(r,z)) \,dr\,dz
-\frac{1}{2} \int_{\Omega}r\nabla \xi_j(r,z)\cdot(B_x(r,z),B_z(r,z)) \,dr\,dz.
\end{align*}
Let us now observe that by construction, we have that both quantities $|(A_x(r,z),A_z(r,z))|$, and $|(B_x(r,z),B_z(r,z))|$ are less than 1. Thus, from the above relation we get that
\begin{align*}
&\int_{\Phi(\Omega\times\mathbb{S}^1)}\chi_{F_{\mu_j}}(x,z)\text{div}_{(x,z)}\varphi(x,z)\, dx\, dz\leq P(E;\Phi(U\times\mathbb{S}^1)) +  \int_{\mathrm{Pr}(\text{supp}(\varphi))}r\left| \nabla \xi_j(r,z)\right| \,dr\,dz,
\end{align*}
where $\mathrm{Pr}(\text{supp}(\varphi))\subset \Omega$ is the projection in $(0,\infty)\times\mathbb{R}^{k-2}$ of the support of $\varphi$, namely
$$
\mathrm{Pr}(\text{supp}(\varphi))=\{(r,z)\in \Omega:\, (\text{supp}(\varphi))_{(r,z)}\neq \emptyset \}.
$$ 
Let us also observe that $\mathrm{Pr}(\text{supp}(\varphi))$ is a compact set in $\Omega$. Recalling that $|\nabla \xi_j|\mathcal{L}^{k-1}\stackrel{*}{\rightharpoonup} |D\xi|$ we immediately get that 
$
r|\nabla \xi_j|\mathcal{L}^{k-1}\stackrel{*}{\rightharpoonup} r|D\xi|.
$
Moreover, since $\xi_j\to\xi$ for $\mathcal{L}^{k-1}$-a.e. $(r,z)\in \Omega$, by the definition of $\mu_j$ we get $\mu_j\to\mu$ for $\mathcal{L}^{k-1}$-a.e. $(r,z)\in \Omega$, which implies that $\chi_{F_{\mu_j}}\to \chi_{F_\mu}$ for $\mathcal{L}^{k}$-a.e. $(x,z)\in \Phi(\Omega\times\mathbb{S}^1)$. Thus, 
\begin{align}\label{eq: final estimate for the P(F mu)}
&\int_{\Phi(\Omega\times\mathbb{S}^1)}\chi_{F_{\mu}}(x,z)\text{div}_{(x,z)}\varphi(x,z)\, dx\, dz=\limsup_{j\to\infty}\int_{\Phi(\Omega\times\mathbb{S}^1)}\chi_{F_{\mu_j}}(x,z)\text{div}_{(x,z)}\varphi(x,z)\, dx\, dz\nonumber\\
&\leq P(E;\Phi(U\times\mathbb{S}^1)) +\limsup_{j\to\infty}  \int_{\mathrm{Pr}(\text{supp}(\varphi))}r\left| \nabla \xi_j(r,z)\right| \,dr\,dz\nonumber \\
&\leq P(E;\Phi(U\times\mathbb{S}^1)) + \, r\left| D \xi\right|(\mathrm{Pr}(\text{supp}(\varphi)))\leq P(E;\Phi(U\times\mathbb{S}^1)) + \, r\left| D \xi\right|(\Omega).
\end{align}
In order to conclude, let us observe that
\begin{align*}
&r|D\xi|(\Omega)=\sup\left\{\int_{\Omega}\psi(r,z)\cdot d(rD_r\xi,rD_{z_i}\xi,\dots,rD_{z_{k-2}}\xi)(r,z):\, \psi\in C^0_c(\Omega;\mathbb{R}^k),\, |\psi|\leq 1 \right\}\\
&\stackrel{\eqref{eq: chian rule fo mu e xi}}{=}\sup\left\{\int_{\Omega}\psi(r,z)\cdot d(rD_r\xi,D_{z_i}\mu,\dots,D_{z_{k-2}}\mu)(r,z):\, \psi\in C^0_c(\Omega;\mathbb{R}^k),\, |\psi|\leq 1 \right\}\\
&\leq P(E;\Phi(\Omega\times\mathbb{S}^1))\leq P(E;\Phi(U\times\mathbb{S}^1)),
\end{align*}
where for the second last inequality we used  \eqref{Dz mu for E}, and \eqref{rDr xi for E}. Combining the above relation with the estimate obtained in \eqref{eq: final estimate for the P(F mu)} we get that
\begin{align}\label{eq: chi Fmu in BV}
&\int_{\Phi(\Omega\times\mathbb{S}^1)}\chi_{F_{\mu}}(x,z)\text{div}_{(x,z)}\varphi(x,z)\, dx\, dz\leq 2P(E;\Phi(U\times\mathbb{S}^1))<\infty.
\end{align}
Taking the sup over all test functions $\varphi\in C^1_c(\Phi(\Omega\times\mathbb{S}^1))$ with $|\varphi|\leq 1$ on the left hand side of the above relation we get that $F_\mu$ has finite perimeter in $\Phi(\Omega\times\mathbb{S}^1)$ for every $\Omega\subset \subset U$ open set. Since the right hand side of \eqref{eq: chi Fmu in BV} does not depend on $\Omega$, by standard arguments we conclude that $F_\mu$ has finite perimeter in $\Phi(U\times\mathbb{S}^1)$. This concludes step four, and we can now proceed to prove relation \eqref{eq: |sigma|=Per (Fmu)}.\\
\textbf{Step 5.} As we said, we are left to prove \eqref{eq: |sigma|=Per (Fmu)}. By standard measure theory, since $|\sigma_{\mu}|$ is a Radon measures on $U$, it is sufficient to show that \eqref{eq: |sigma|=Per (Fmu)} holds true for every open set $A\subset U$. We start proving that
\begin{align}\label{eq: |sigma| < Per(Fmu)}
|\sigma_{\mu}|(A)\leq P(F_\mu;\Phi(A\times\mathbb{S}^1))\,\quad\forall\,A\subset U\textnormal{ open}.
\end{align}
Let $\varphi\in C^0_c(A;\mathbb{R}^k)$ with $|\varphi|\leq 1$, and let $V\subset\subset U$ be an open set such that $\textnormal{supp} (\varphi)\subset V$. Then, since we proved that $F_\mu$ has finite perimeter in $\Phi(U\times\mathbb{S}^1)$, and since by construction $(F_\mu)_{(r,z)}$ is $\mathcal{H}^1$-equivalent to a connected arc in $\partial B(r)$ for $\mathcal{L}^{k-1}$-a.e. $(r,z)\in A$, we can apply \eqref{eq: the real calc lemmata for Fmu}  thus obtaining 
\begin{align*}
&\int_A \varphi(r,z)\cdot d\sigma_{\mu}(r,z)=\int_V \varphi(r,z)\cdot d\sigma_{\mu}(r,z)\\
&\stackrel{\eqref{eq: the real calc lemmata for Fmu}}{=}\int_{\partial^*F_\mu\cap \Phi(V\times\mathbb{S}^1)} \varphi(|x|,z)\cdot\nu^{F_\mu}_{\mathsf{c}}(x,z)\,d\mathcal{H}^{k-1}(x,z)\leq P(F_\mu;\Phi(A\times\mathbb{S}^1))<\infty,
\end{align*}
where in the first inequality we used Schwartz inequality. Passing to the sup in the left hand side among all $\varphi\in C^0_c(A;\mathbb{R}^k)$ with $|\varphi|\leq 1$ , we prove \eqref{eq: |sigma| < Per(Fmu)}. Let us now prove the reverse inequality, namely
\begin{align}\label{eq: |sigma| > Per(Fmu)}
|\sigma_{\mu}|(A)\geq P(F_\mu;\Phi(A\times\mathbb{S}^1)).
\end{align}
\noindent
Recall now the definition of the Borel vector field $\bar{\nu}^{F_\mu}_{\mathsf{c}}:(0,\infty)\times\mathbb{R}^{k-2}\to \mathbb{R}^{k}$ that was given in \eqref{def: bar(nu)}.
\noindent
Thus, denoting by $d\sigma_{\mu}/ d|\sigma_{\mu}|:U\to\mathbb{S}^{k-1}$ the polar decomposition of $\sigma_{\mu}$, we get
\begin{align*}
P&(F_\mu;\Phi(A\times\mathbb{S}^1))=\int_{\partial^*F_\mu\cap \Phi(A\times\mathbb{S}^1)}\!\!\!\! 1\, d\mathcal{H}^{k-1}(x,z)= \int_{\partial^*F_\mu\cap \Phi(A\times\mathbb{S}^1)}\bar{\nu}^{F_\mu}_{\mathsf{c}}(|x|,z)\cdot\nu^{F_\mu}_{\mathsf{c}}(x,z)\, d\mathcal{H}^{k-1}(x,z)
\\&\stackrel{\eqref{eq: the real calc lemmata for Fmu}}{=}\int_{A} \bar{\nu}^{F_\mu}_{\mathsf{c}}(r,z)\cdot d\sigma_{\mu}(r,z)= \int_{A} \bar{\nu}^{F_\mu}_{\mathsf{c}}(r,z)\cdot\frac{d\sigma_{\mu}}{d|\sigma_{\mu}|}(r,z)\,d|\sigma_{\mu}|(r,z)
\leq \int_{A} 1\,d|\sigma_{\mu}|(r,z)=|\sigma_{\mu}|(A),
\end{align*}
where in the last inequality we used the Schwartz inequality. This concludes the proof of \eqref{eq: |sigma| > Per(Fmu)} which together with \eqref{eq: |sigma| < Per(Fmu)} gives \eqref{eq: |sigma|=Per (Fmu)}. 
\end{proof}

\begin{remark}
Let us observe that, as a consequence of \eqref{eq: |sigma|=Per (Fmu)}, and thanks to the argument used to prove it, we get that
\begin{align}\label{eq: dsigma/d|sigma =| nu^Fmu_circ}
&\frac{d\sigma_{\mu}}{d|\sigma_{\mu}|}(r,z)=\bar{\nu}^{F_\mu}_{\mathsf{c}}(r,z)\quad \textnormal{for }|\sigma_{\mu}|\textnormal{-a.e. }(r,z)\in U,
\end{align}
where $\bar{\nu}^{F_\mu}_{\mathsf{c}}$ was defined in \eqref{def: bar(nu)}.
\end{remark}


\begin{remark}\label{formula per Fmu}
Another consequence of relation \eqref{eq: |sigma|=Per (Fmu)} is the following formula for the perimeter of $F_\mu$, namely for every $B\subset U$ Borel we have that
\begin{align*}
&P(F_\mu;\Phi(B\times\mathbb{S}^1))=2\int_{B}\sqrt{1+ \frac{1}{4}\left| r\frac{\partial}{\partial r}\xi(r,z) \right|^2+ \frac{1}{4}\left|\nabla_z\, \mu(r,z)\right|^2}\,dr\, dz + \left|\left(D^s_r\xi,D_z^s\mu \right)\right|(B),
\end{align*}
where by $\frac{\partial}{\partial r}\xi$, and $\nabla_z \mu$ we denote the first component of $D^a \xi$, and the last $(k-2)$ components of $D^a \mu$, respectively.
\end{remark}

\section{Characterisation of equality cases}
\label{sec: characterisation of equality cases}

\begin{proof}[Proof of Theorem \ref{Eq. cases per. ineq.}]
Let us prove \eqref{eq: circular per ineq}. Indeed, 
\begin{align*}
|\sigma_{\mu}|(B)&\stackrel{\eqref{eq: dsigma/d|sigma =| nu^Fmu_circ}}{=}\int_{B}\bar{\nu}^{F_\mu}_{\mathsf{c}}(r,z)\cdot \frac{d\sigma_{\mu}}{d|\sigma_{\mu}|}(r,z)\,d|\sigma_{\mu}|(r,z)=\int_{B}\bar{\nu}^{F_\mu}_{\mathsf{c}}(r,z)\cdot d\sigma_{\mu}(r,z)\\
&\stackrel{\eqref{eq: the real calc lemmata weaker}}{\leq} \int_{\partial^*E\cap \Phi(B\times\mathbb{S}^1)}\bar{\nu}^{F_\mu}_{\mathsf{c}}(|x|,z)\cdot \nu^E_{\mathsf{c}}(x,z)\, d\mathcal{H}^{k-1}(x,z)
\leq P(E;\Phi(B\times\mathbb{S}^1)),
\end{align*}
where for the last inequality we used Schwartz inequality. This, together with \eqref{eq: |sigma|=Per (Fmu)} proves \eqref{eq: circular per ineq}. Immediately from the above chain of inequalities we get that conditions a) and b) are sufficient to have $P(F_\mu;\Phi(B\times\mathbb{S}^1))=P(E;\Phi(B\times\mathbb{S}^1))$. Indeed, by condition a) we get an equality sign in \eqref{eq: the real calc lemmata weaker}, while by condition b) we get the equality sign in the last inequality appearing above. Vice versa, let us assume that $P(F_\mu;\Phi(B\times\mathbb{S}^1))=P(E;\Phi(B\times\mathbb{S}^1))$. Then, by the equality sign in \eqref{eq: the real calc lemmata weaker} we get that condition a) is satisfied. Moreover, by imposing the equality sign also in the last inequality appearing in the above relations we get that,
\begin{align*}
\bar{\nu}^{F_\mu}_{\mathsf{c}}(|x|,z)\cdot \nu^E_{\mathsf{c}}(x,z)=1\quad \textnormal{for }\mathcal{H}^{k-1}\textnormal{-a.e. }(x,z)\in \partial^*E\cap \Phi(B\times\mathbb{S}^1).
\end{align*}
Thus, up to remove a set $N\subset \partial^*E$ with $\mathcal{H}^{k-1}(N)=0$ we have that
\begin{align*}
\bar{\nu}^{F_\mu}_{\mathsf{c}}(|x|,z)\cdot \nu^E_{\mathsf{c}}(x,z)=1\quad \textnormal{for every }(x,z)\in (\partial^*E\setminus N)\cap \Phi(B\times\mathbb{S}^1),
\end{align*}
which recalling Proposition \ref{real magic property}, is equivalent to say that for every $(r,z)\in B$ such that both $(\partial^*F_\mu)_{(r,z)}\neq \emptyset$, and $(\partial^*E\setminus N)_{(r,z)}\neq \emptyset$, we have that $\nu^E_{\mathsf{c}}(x,z)=\bar{\nu}^{F_\mu}_{\mathsf{c}}(|x|,z)$ for every $x\in (\partial^*E\setminus N)_{(r,z)}$. This directly implies condition b), and so we conclude the proof.
\end{proof}

\begin{remark}
Let $B\subset U$ be a Borel set such that we are in an equality case for \eqref{eq: circular per ineq} w.r.t. the set $B$. Let us stress that calling with $\tilde{B}\subset B$ the set 
$$
\tilde{B}:=\left\{(r,z)\in B:\, (\partial^*F_\mu)_{(r,z)}\neq \emptyset,\textnormal{ and }(\partial^*E\setminus N)_{(r,z)}\neq \emptyset \right\},
$$
we have that
\begin{align}\label{eq:final remark}
P(E;\Phi(\tilde{B}\times\mathbb{S}^1))=P(E;\Phi(B\times\mathbb{S}^1))=P(F_\mu;\Phi(B\times\mathbb{S}^1))=P(F_\mu;\Phi(\tilde{B}\times\mathbb{S}^1)).
\end{align}
Indeed, if we consider the following two sets
\begin{align*}
&B_1:=\left\{(r,z)\in B:\, (\partial^*F_\mu)_{(r,z)}=\emptyset\right\},\\
&B_2:=\left\{(r,z)\in B:\,(\partial^*E\setminus N)_{(r,z)}= \emptyset  \right\},
\end{align*}
we get that $B\setminus \tilde{B}=B_1\cup B_2$ and
\begin{align*}
0=P(F_\mu;\Phi(B_1\times\mathbb{S}^1))=P(E;\Phi(B_1\times\mathbb{S}^1))=|\sigma_\mu|(B_1),\\
0=P(E;\Phi(B_2\times\mathbb{S}^1))=P(F_\mu;\Phi(B_2\times\mathbb{S}^1))=|\sigma_\mu|(B_2)
\end{align*}
from which we easily deduce \eqref{eq:final remark}.
\end{remark}

\section{Steiner symmetrisation setting}\label{sec: Steiner setting}
In this section we will present the results obtained for the circular symmetrisation, but for the Steiner setting. We will present the results without proofs since they can be obtained  by adapting the arguments used in the previous sections. 

\medskip
Let $k \in \mathbb{N}$, with $k \geq 2$.
We will decompose $\mathbb{R}^k$ as $\mathbb{R}^{k-1} \times \mathbb{R}$, 
and we will write $(x',y) \in \mathbb{R}^k$, with $x' \in \mathbb{R}^{k-1}$ and $y \in \mathbb{R}$. We are now going to define the Steiner symmetral of a Borel set in $\mathbb{R}^k$ with respect to the hyperplane 
\( \{ (x', y) \in \mathbb{R}^{k} : y = 0 \} = \mathbb{R}^{k-1}\times\{ 0 \}. \)
For every Borel set $E \subset \mathbb{R}^{k}$ we define
\[
E_{x'} := \{ y \in \mathbb{R}:\, (x',y) \in E \} 
\qquad \text{ for every } x' \in \mathbb{R}^{k-1},
\]
Let now $v : \mathbb{R}^{k-1} \to [0, \infty)$ be a Lebesgue measurable function. We will say that $E$ is \textit{$v$-distributed} if 
\[
v (x') = \mathcal{H}^1 (E_{x'}), \quad \text{ for $\mathcal{L}^{k-1}$-a.e. } x' 
\in \mathbb{R}^{k-1}.
\]
Given a Lebesgue measurable function $v: \mathbb{R}^{k-1} \to [0, \infty)$ we define the set $F[v]\subset \mathbb{R}^k$ as 
\begin{align}\label{def: F[v]}
F[v] := \left\{ (x', y) \in  \mathbb{R}^{k} :\,  
 |y|<\frac{1}{2} v(x') \right\}.
\end{align}

\begin{remark} \label{Steiner rem used for monotonicity}
Note that by definition of $F[v]$, we have 
\[
(x', y) \in F[v] \, \Longrightarrow \,
(x', z) \in F[v]
\quad \forall \, z\in\mathbb{R} \textnormal{ such that }   |z|\leq |y|.
\]
\end{remark}

If $E \subset \mathbb{R}^{k}$ is a $v$-distributed Borel set, we say that $F[v]$ is the \textit{Steiner  symmetral} of $E$
with respect to the hyperplane 
\( \{ (x', y) \in \mathbb{R}^{k} : y = 0 \} \).

\subsection{Properties of $F[v]$ and $v$}
Next result is the Steiner counterpart of Lemma \ref{lem: decreasing densities}.
\begin{lemma}
Let $v  : \mathbb{R}^{k-1} \to [0, \infty)$ be a Lebesgue measurable. Let $x' \in  \mathbb{R}^{k-1}$. Then, the functions 
\[
z \mapsto \theta_* (F[v],  ( x',z))
\quad 
\text{ and } 
\quad
z \mapsto \theta^* (F[v],  (x' ,z ))
\]
are even in $(-\infty,\infty)$ and non increasing in $[0, \infty)$.
\end{lemma}
\noindent
The following result is the Steiner counterpart of Proposition \ref{magic property}.
\begin{proposition} \label{Steiner magic property}
Let $v :\mathbb{R}^{k-1}\to [0, \infty)$ be a Lebesgue measurable function such that $F[v]$ is a set of locally finite perimeter. Suppose that $(x',y) \in \partial^* F[v]$.
Then
\begin{align}\label{formula symmetric normals in slices Steiner}
\nu^{F[v]} (x', z) = \nu^{F[v]}(x',y) 
\end{align}
for every $z \in \left[ \min\{ y, 0 \}, \max\{ y, 0 \} \right]$ such that
$(x', z) \in ( \partial^* F[v] )_{x'}$. 
\end{proposition}
\noindent
The following remark is the Steiner counterpart of Remark \ref{rem: nu ^F_mu is symmetric wrt to x2=0} (compare it with \cite[Remark~2.5]{barchiesicagnettifusco}).
\begin{remark}\label{rem: nu ^F[v] is symmetric wrt to the Steiner hyperplane }
Let us notice that, by symmetry of the set $F[v]$ w.r.t. the hyperplane $\{ y=0 \}\subset \mathbb{R}^k$, the following property holds true: for every $(x',y)\in \partial^*F[v]$ we have that $(x',-y)\in \partial^*F[v]$ and 
$$
\nu^{F_\mu}(x',-y)= (\nu^{F[v]}_1(x',y),\dots,\nu^{F[v]}_{k-1}(x',y),-\nu^{F[v]}_k(x',y)).
$$ 
\end{remark}
\noindent
The following result represents the Steiner counterpart of Proposition \ref{real magic property}.
\begin{proposition} \label{Steiner real magic property}
Let $v :\mathbb{R}^{k-1}\to [0, \infty)$ be a Lebesgue measurable function such that $F[v]$ is a set of locally finite perimeter. Then, for every $x' \in  \R^{k-1} $
such that $( \partial^* F[v] )_{x'} \neq \emptyset$, the functions  
\begin{align}\label{eq: the Steiner components of nu^F[v]}
y \mapsto \nu^{F[v]}_i ( x', y)\quad \textnormal{for } i=1,\dots,k-1,\qquad
y \mapsto | \nu^{F[v]}_k ( x', y)|,
\end{align}
are constant in $( \partial^* F[v])_{x'}$.
\end{proposition}
\noindent
Given $E\subset \R^k$ set of locally finite perimeter, we set
\begin{align}\label{eq: the Steiner version of nu^E}
\nu^E_{\mathsf{s}}(x',y):=(\nu^E_1(x',y),\dots, \nu^E_{k-1}(x',y),|\nu^{E}_k(x',y)|),\quad
\textnormal{for $\mathcal{H}^k$-a.e. $(x',y)\in\partial^*E$}.
\end{align}
\noindent
Thanks to Proposition \ref{Steiner real magic property}, we set 
\begin{align}\label{def: Steiner bar(nu)}
\bar{\nu}^{F[v]}_\mathsf{s}(x'):=
\begin{cases}
\nu^{F[v]}_\mathsf{s}(x',y) &\mbox{ if }(\partial^*F[v])_{x'}\neq \emptyset, \textnormal{ and }y\in (\partial^*F[v])_{x'}, \\
0 &\mbox{ otherwise.}
\end{cases}
\end{align}
\noindent
Next result is the Steiner counterpart of Lemma \ref{lem: mu e xi in BV loc} (compare this result with \cite[Lemma 3.1]{ChlebikCianchiFuscoAnnals05}). 
\begin{lemma}\label{lem:Steiner v in BV}
Let $v  : \R^{k-1} \to [0, \infty)$ be a Lebesgue measurable function, and let $E\subset\R^k$ be a $v$-distributed set of finite perimeter and finite volume. Then, $v\in BV(\mathbb{R}^{k-1})$. In addition, $|D_i v|$ is a finite Radon measure 
on $\R^{k-1}$ for every $i=1,\dots, k-1$, and  
for every Borel set $B \subset \R^{k-1}$ we have
\begin{align}
&\int_{B} \varphi(x') \, d D_{i} v (x')
= \int_{\partial^*E \cap  (B \times \mathbb{R}) }  
\varphi (x') \, \nu_{i}^{E} (x', y) \, d \mathcal{H}^{k-1} (x',y), \label{Di v for E} 
\end{align}
for every $i=1,\dots, k-1$, and for every bounded Borel function $\varphi: B \to \mathbb{R}$. Moreover, let $\sigma_{v}$ be the $\mathbb{R}^k$-valued Radon measure on $\mathbb{R}^{k-1}$ defined as
\begin{align}\label{ def of sigma_s}
\sigma_{v}(B):=\int_{B}d(D_1 v,\dots,D_{k-1}v,2\mathcal{L}^{k-1}\mres \{ v>0\} )(x'),\quad \forall\,B\subset \mathbb{R}^{k-1} \textnormal{ Borel}.
\end{align}
Then, for every Borel set $B\subset \mathbb{R}^{k-1}$ we get 
\begin{align}\label{eq: Steiner the real calc lemmata weaker}
\int_B\varphi(x')\cdot d\sigma_{v}(x')\leq\int_{\partial^*E\cap (B\times\mathbb{R})}\varphi(x')\cdot\nu^E_{\mathsf{s}}(x',y)\,d\mathcal{H}^{k-1}(x',y),
\end{align}
for every bounded Borel function $\varphi: B \to \mathbb{R}^k$ with non-negative last component, where $\nu^E_{\mathsf{s}}$ was defined in \eqref{eq: the Steiner version of nu^E}. In particular, equality sign holds true in \eqref{eq: Steiner the real calc lemmata weaker} if and only if $(E)_{x'}$ is $\mathcal{H}^1$-equivalent to a segment, for $\mathcal{L}^{k-1}$-a.e. $x'\in B$.
\end{lemma}

\begin{remark}\label{rem: Steiner weak calculation lemmata specifically for Fmu}
Under the assumptions of the above lemma, let $B\subset \mathbb{R}^{k-1}$ be a Borel set, and let $E\subset \mathbb{R}^{k}$ be a $v$-distributed set of finite perimeter and finite volume such that $(E)_{x'}$ is $\mathcal{H}^1$-equivalent to a segment for $\mathcal{L}^{k-1}$-a.e. $x'\in B$. Then, we get 
\begin{align}\label{eq:  Steiner the real calc lemmata for Fmu}
\int_B\varphi(x')\cdot d\sigma_{v}(x')=\int_{\partial^*E\cap B\times\mathbb{R}}\varphi(x')\cdot\nu^E_{\mathsf{s}}(x',y)\,d\mathcal{H}^{k-1}(x',y),
\end{align}
for every bounded Borel function $\varphi: B \to \mathbb{R}^k$.
\end{remark}

\noindent
The next result is the Steiner counterpart of Lemma  \ref{lem: P(F[v])=|sigma|} (compare this result with \cite[Corollary 3.4]{CagnettiColomboDePhilippisMaggiSteiner}, and with \cite[Lemma 3.5]{ChlebikCianchiFuscoAnnals05}).
\begin{lemma}
Let $v  : \R^{k-1} \to [0, \infty)$ be a Lebesgue measurable function, and let $E\subset\R^k$ be a $v$-distributed set of finite perimeter and finite volume. Then, the set $F[v]\subset \mathbb{R}^{k}$ defined in \eqref{def: F[v]} is a set of finite perimeter and finite volume. Moreover,
\begin{align}\label{eq: |sigma_s|=Per (F[v])}
|\sigma_{v}|(B)=P(F[v];B\times\mathbb{R}),\quad \forall\,B\subset \mathbb{R}^{k-1} \textnormal{ Borel},
\end{align}
where $\sigma_{v}$ is the Radon measure defined in \eqref{ def of sigma_s}.
\end{lemma}

\subsection{ Characterisation of equality cases}
Next result is the Steiner counterpart of Theorem \ref{Eq. cases per. ineq.} (compare this result with  \cite[Theorem 1.1, Lemma 3.4]{ChlebikCianchiFuscoAnnals05}). 
\begin{theorem}\label{Steiner Eq. cases per. ineq.}
Let $v  : \mathbb{R}^{k-1}\to [0, \infty)$ be a Lebesgue measurable function, and let $E \subset \R^k$ be a $v$-distributed set of finite perimeter and finite volume. Then, 
\begin{align}\label{eq: Steiner per ineq}
P (F[v]; B \times \mathbb{R}) \leq P (E; B \times \mathbb{R}),\quad \forall\,B \subset  \mathbb{R}^{k-1} \textnormal{ Borel}.
\end{align} 
Moreover, equality holds in \eqref{eq: Steiner per ineq} for some Borel set $B\subset \mathbb{R}^{k-1}$ if and only if both the following two conditions are satisfied.
\begin{itemize}
\vspace{.1cm}
\item[a)]For $\mathcal{L}^{k-1}$-a.e. $x'\in B$ we have that $(E)_{x'}$ is $\mathcal{H}^1$-equivalent to a segment.
\vspace{.1cm}
\item[b)]There exists $N\subset \partial^*E$ with $\mathcal{H}^{k-1}(N)=0$, with the property that \textbf{for every} $x'\in B$ such that $( \partial^* E\setminus N )_{x'} \neq \emptyset$, and $(\partial^*F[v])_{x'} \neq \emptyset$, we have that
\begin{align*}\label{eq: magic property for E}
\nu^{E}_{\mathsf{s}}(x',y)=\bar{\nu}^{F[v]}_{\mathsf{s}}(x')\quad \forall y\in (\partial^*E\setminus N)_{x'},
\end{align*}
where $\bar{\nu}_{\mathsf{s}}^{F[v]}$ was defined in \eqref{def: Steiner bar(nu)}.
\end{itemize}
\end{theorem}


\begin{remark}
By definition of $\nu^E_{\mathsf{s}}$, condition b) of the above result implies that \textbf{for every} $x'\in B$ such that $( \partial^* E\setminus N )_{x'} \neq \emptyset$, and $(\partial^*F[v])_{x'} \neq \emptyset$ the functions
\begin{align*}
y \mapsto \nu^E_i(x',y)\quad\textnormal{for } i=1,\dots,k-1,\qquad 
y \mapsto |\nu^{E}_k(x',y)|,
\end{align*}
are constant in $(\partial^*E\setminus N)_{x'}$.
\end{remark}

\begin{remark}
Let us point out that if $B=\mathbb{R}^{k-1}$, condition \textnormal{a)} of the above result coincides with \cite[(1.7) of Theorem 1.1]{ChlebikCianchiFuscoAnnals05}, while condition \textnormal{b)} is a refinement of \cite[(1.8) of Theorem 1.1]{ChlebikCianchiFuscoAnnals05}.
\end{remark}

\def\cprime{$'$}


\begin{thebibliography}{10}


\bibitem{AFP}
{\sc L.~Ambrosio, N.~Fusco, and D.~Pallara}, {\em Functions of bounded
  variation and free discontinuity problems}, Oxford Mathematical Monographs,
  The Clarendon Press, Oxford University Press, New York, 2000.



\bibitem{barchiesicagnettifusco}
{\sc M.~Barchiesi, F.~Cagnetti, and N.~Fusco}, {\em Stability of the {S}teiner
  symmetrization of convex sets}, J. Eur. Math. Soc., 15 (2013),
  pp.~1245--1278.


  
  
    
  
\bibitem{bogelainduzaarfusco2017}  
{\sc V. B\"{o}gelein, F. Duzaar, N. Fusco}, {\em A quantitative isoperimetric inequality on the sphere},
Adv. Calc. Var. (3), 10 (2017), pp.~223--265.
 


\bibitem{CagnettiColomboDePhilippisMaggiSteiner}
{\sc F.~Cagnetti, M.~Colombo, G.~De~Philippis, and F.~Maggi}, {\em Rigidity of
  equality cases in {S}teiner's perimeter inequality}, Anal. PDE, 7 (2014),
  pp.~1535--1593.

\bibitem{ccdpmGAUSS}
\leavevmode\vrule height 2pt depth -1.6pt width 23pt, {\em Essential
  connectedness and the rigidity problem for {G}aussian symmetrization}, J.
  Eur. Math. Soc. (JEMS), 19 (2017), pp.~395--439.

\bibitem{CagnettiPeruginiStoger} 
{\sc
F.~Cagnetti, M.~Perugini, D.~St\"oger} 
{\em Rigidity for perimeter inequality under spherical symmetrisation}. 
Calc. Var. Partial Differential Equations 59 (2020), no. 4, Paper No 139, 53 pp. 




\bibitem{ChlebikCianchiFuscoAnnals05}
{\sc M.~Chleb{\'{\i}}k, A.~Cianchi, and N.~Fusco}, {\em The perimeter
  inequality under {S}teiner symmetrization: cases of equality}, Ann. of Math.
  (2), 162 (2005), pp.~525--555.


\bibitem{cianchifusco2}
{\sc A.~Cianchi and N.~Fusco}, {\em Steiner symmetric
  extremals in {P}\'{o}lya-{S}zeg\"{o} type inequalities}, Adv. Math., 203
  (2006), pp.~673--728.


\bibitem{DeGiorgi58ISOP}
{\sc E.~De~Giorgi}, {\em Sulla propriet\`a isoperimetrica dell'ipersfera, nella
  classe degli insiemi aventi frontiera orientata di misura finita}, Atti
  Accad. Naz. Lincei. Mem. Cl. Sci. Fis. Mat. Nat. Sez. I (8), 5 (1958),
  pp.~33--44.

\bibitem{DeGiorgiSelected}
\leavevmode\vrule height 2pt depth -1.6pt width 23pt, {\em Selected papers},
  Springer Collected Works in Mathematics, Springer, Heidelberg, 2013.
\newblock [Author name on title page: Ennio Giorgi], Edited by Luigi Ambrosio,
  Gianni Dal Maso, Marco Forti, Mario Miranda and Sergio Spagnolo, Reprint of
  the 2006 edition.




%
%

\bibitem{GMSbook1}
{\sc M.~Giaquinta, G.~Modica, and J.~Soucek}, {\em Cartesian currents in the
  {C}alculus of {V}ariations. I. Cartesian currents}, vol.~37 of Ergebnisse der
  Mathematik und ihrer Grenzgebiete. 3. Folge. A Series of Modern Surveys in
  Mathematics, Springer-Verlag, Berlin, 1998.



\bibitem{kawohl_book_85}
{\sc B.~Kawohl}, {\em Rearrangements and convexity of level sets in {PDE}},
  vol.~1150 of Lecture Notes in Mathematics, Springer-Verlag, Berlin, 1985.



\bibitem{maggiBOOK}
{\sc F.~Maggi}, {\em Sets of finite perimeter and geometric variational
  problems}, vol.~135 of Cambridge Studies in Advanced Mathematics, Cambridge
  University Press, Cambridge, 2012.
\newblock An introduction to {G}eometric {M}easure {T}heory.




 

\bibitem{Perugini}
{\sc M.~Perugini}, {\em Rigidity of Steiner's inequality for the anisotropic perimeter}
Ann. Sc. Norm. Super. Pisa Cl. Sci., (5)
Vol. XXIII (2022), pp. 1921--2001.


  
\bibitem{polya50} {\sc G.~P\'olya}, {\em Sur la sym\'etrisation circulaire}. (French) 
C. R. Acad. Sci. Paris 230, (1950), pp~25--27. 

\bibitem{polyaszego51} {\sc G.~P\'olya, G. Szeg\"{o}}, 
{\em Isoperimetric {I}nequalities in {M}athematical {P}hysics},
Annals of Mathematics Studies, No. 27, 
Princeton University Press, Princeton, N. J., 1951.





\bibitem{Simon83} {\sc L.~Simon}, 
{\em Lectures on geometric measure theory}, 
Proceedings of the Centre for Mathematical Analysis, Australian
National University, 3, Centre for Mathematical Analysis, Canberra, 1983.





\bibitem{Volpert}
{\sc A.~I. Vol\cprime~pert}, {\em Spaces {${\rm BV}$} and quasilinear
  equations}, Mat. Sb. (N.S.), 73 (115) (1967), pp.~255--302.

\end{thebibliography}
\end{document}